\documentclass[11pt]{amsart}
\newcommand{\ob}[1]{#1}

\usepackage{paralist, amsmath, amsthm, amssymb, color, graphicx, hyperref}
\usepackage[textwidth=20mm,textsize=footnotesize]{todonotes}
\usepackage[margin=1in]{geometry}
\DeclareGraphicsExtensions{.jpg,.pdf,.eps}
\graphicspath{{./Figures/}}
\usepackage{amssymb}
\usepackage{setspace}
\usepackage{mathtools}
\usepackage{color}
\usepackage{stmaryrd}

\usepackage{xfrac}

\usepackage{mathrsfs}
\usepackage{amscd}
\usepackage{dsfont}
\usepackage{tikz}
\usepackage{verbatim}
\usepackage{enumerate}

\theoremstyle{plain}
\newtheorem{thm}{Theorem}[section] 
\newtheorem{lemma}[thm]{Lemma}
\newtheorem{prop}[thm]{Proposition}
\newtheorem{cor}[thm]{Corollary}

\theoremstyle{definition}
\newtheorem{defn}[thm]{Definition} 
\newtheorem{definition}[thm]{Definition} 
\newtheorem{exmp}[thm]{Example} 
\newtheorem{example}[thm]{Example}
\newtheorem{remark}[thm]{Remark}
\newtheorem{note}[thm]{Note}

\newcommand{\mscr}[1]{\mathscr{#1}}
\newcommand{\mfrk}[1]{\mathfrak{#1}}

\newcommand{\mC}{\mscr{C}}
\newcommand{\fC}{\mfrk{C}}

\newcommand{\ZZ}{\mathbb{Z}}
\newcommand{\RR}{\mathbb{R}}

\newcommand{\CC}{\mathbb{C}}
\newcommand{\QQ}{\mathbb{Q}}

\newcommand{\mO}{\mathcal{O}}

\newcommand{\al}{\alpha}
\newcommand{\be}{\beta}
\newcommand{\ga}{\gamma}
\newcommand{\de}{\delta}
\newcommand{\eps}{\epsilon}

\newcommand{\Om}{\Omega}

\newcommand{\ONE}{\mathds{1}}

\usepackage{mathtools}
\newcommand{\orient}[1]{\overleftrightarrow{#1}}

\newcommand{\defeq}{\vcentcolon=}
\newcommand{\chrom}{\chi}
\DeclareMathOperator{\Apoly}{\mathbf{A}}
\DeclareMathOperator{\Bpoly}{\mathbf{B}}
\newcommand{\Potts}{P}

\DeclareMathOperator{\rk}{rk}

\DeclareMathOperator{\Null}{Null}
\DeclareMathOperator{\exterior}{ext}
\DeclareMathOperator{\interior}{int}

\DeclareMathOperator{\Stab}{Stab}
\DeclareMathOperator{\incidence}{R}
\DeclareMathOperator{\Orient}{Orient}

\newcommand{\xx}{\textbf{x}}
\newcommand{\psim}{\overset{+}{\sim}}

\newcommand{\Id}{\textrm{Id}}

\DeclareMathOperator{\comp}{\textrm{c}}

\newcommand{\ov}{\overline}
\newcommand{\un}{\underline}
\newcommand{\wt}{\widetilde}

\newcommand{\ds}{\displaystyle}
\renewcommand{\ni}{\noindent}

\newcommand{\bu}{\mathbf{u}}
\newcommand{\bt}{\mathbf{t}}
\newcommand{\ba}{\mathbf{a}}
\newcommand{\bx}{\mathbf{x}}
\newcommand{\bbA}{\mathbb{A}}
\newcommand{\bbB}{\mathbb{B}}

\newcommand{\fig}[3]{\begin{figure}[h!]\begin{center}\includegraphics[#1]{#2}\end{center}\caption{#3}\label{fig:#2}\end{figure}}

\title{Tutte polynomials for regular oriented matroids}

\author{Jordan Awan and Olivier Bernardi}
\thanks{OB acknowledges support from NSF grant DMS-2154242. JA acknowledges support from GAANN fellowship  P200a120007 from the U.S. Department of Education.}
\date{\today}
\address{Department of Statistics, Purdue University,  West Lafayette, IN 47907, USA}
\email{jawan@purdue.edu}

\address{Department of Mathematics, Brandeis University, Waltham, MA 02453, USA}
\email{bernardi@brandeis.edu}

\begin{document}
\setcounter{tocdepth}{2}

\begin{abstract}
The Tutte polynomial is a fundamental invariant of graphs and matroids. In this article, we define a generalization of the Tutte polynomial to oriented graphs and regular oriented matroids. 
To any regular oriented matroid $N$, we associate a polynomial invariant $\Apoly_N(q,y,z)$, which we call the \emph{$A$-polynomial}.
The $A$-polynomial has the following interesting properties among many others: 
\begin{itemize}
\item a specialization of $\Apoly_N$ gives the Tutte polynomial of the underlying unoriented matroid $\un N$,
\item when the oriented matroid $N$ corresponds to an unoriented matroid (that is, when the elements of the ground set come in pairs with opposite orientations), the invariant $\Apoly_N$ is equivalent to the Tutte polynomial of this unoriented matroid (up to a change of variables),
\item the invariant $\Apoly_N$ detects, among other things, whether $N$ is acyclic and whether $N$ is totally cyclic.
\end{itemize}
We explore various properties and specializations of the $A$-polynomial. We show that some of the known properties of the Tutte polynomial of matroids can be extended to the $A$-polynomial of regular oriented matroids. 
For instance, we show that a specialization of $\Apoly_N$ counts all the acyclic orientations obtained by reorienting some elements of $N$, according to the number of reoriented elements. 

Let us mention that in a previous article we had defined an invariant of oriented graphs that we called the \emph{$B$-polynomial}, which is also a generalization of the Tutte polynomial. However, the $B$-polynomial of an oriented graph $N$ is not equivalent to its $A$-polynomial, and the $B$-polynomial cannot be extended to an invariant of regular oriented matroids.
\end{abstract}

\maketitle

\section{Introduction}
The Tutte polynomial is a fundamental invariant of (unoriented) graphs and matroids. There is a vast and rich literature about the Tutte polynomial; see for instance~\cite{Brylawski:Tutte-poly,Welsh:PottsTutte} or~\cite[Chapter 9]{Gordon:matroids} for an introduction. In this article we investigate a generalization of the Tutte polynomial to regular oriented matroids, that we name the \emph{$A$-polynomial}. \\

Recall that a \emph{regular oriented matroid} is an oriented matroid representable over $\RR$ by a totally unimodular matrix. In particular, the \emph{graphic oriented matroids} which are the oriented matroids coming from directed graphs are regular oriented matroids. 
Hence, the $A$-polynomial gives in particular a directed graph invariant generalizing the Tutte polynomial. \\

The $A$-polynomial of a regular oriented matroid is a polynomial in three variables which is defined as follows (see Section~\ref{DefA} for more details). Let $N$ be a regular oriented matroid with ground set $A$ and circuit set $\mfrk{C}$. For a positive integer~$q$, a \emph{$q$-coflow} of $N$ is a map $f:A\rightarrow \ZZ/q\ZZ$ such that for all circuits $C=(C^+,C^-)\in \mfrk{C}$, 
$$\sum_{a\in C^+}f(a)-\sum_{a\in C^-}f(a)=0 \quad (\text{in $\ZZ/q\ZZ$}).$$
The $A$-polynomial of $N$, $\Apoly_N(q,y,z)$, is then defined as the unique polynomial in $q,y,z$ such that for any odd positive integer~$q$,
\begin{equation}\label{eq:def-A}
\Apoly_N(q,y,z)= \sum_{f~q\textrm{-coflow}}
y^{\#a\in A,~ f(a)\in \left\{1,2,\ldots,\frac{q-1}{2}\right\}}\,
z^{\#a\in A,~ -f(a)\in \left\{1,2,\ldots,\frac{q-1}{2}\right\}},
\end{equation}
where for an element $x\in \ZZ/q\ZZ$ and a set of integers~$S$, we write $x\in S$ to mean that $x=s+q\ZZ$ for some $s\in S$.
The existence of the polynomial $\Apoly_N(q,y,z)$ satisfying the above definition will be established in Section~\ref{Ehrhart1} using Ehrhart theory.\\

Let us now state the relations between the $A$-polynomial and the Tutte polynomial. It is classical to consider unoriented graphs as a special case of directed graphs, by identifying unoriented graphs with the directed graphs having arcs coming in pairs with opposite orientations. This is represented in Figure~\ref{fig:graph-as-digraph}. Similarly, we will identify regular unoriented matroids to the regular oriented matroids such that the ground set elements come in pairs with opposite orientations. This gives a natural way to define the $A$-polynomial of a regular unoriented matroid. In Section~\ref{Potts_links} we prove that the $A$-polynomial of a regular unoriented matroid is equal to its Tutte polynomial, up to a change of variables. In this sense, the $A$-polynomial is a legitimate generalization of the Tutte polynomial. \\

There are two additional relations between the $A$-polynomial and the Tutte polynomial. 
First, for any regular unoriented matroid $M$, the Tutte polynomial of $M$ is equivalent to the sum of the $A$-polynomials of all oriented matroids obtained by orienting $M$. Second, for any regular oriented matroid $N$, a specialization of the $A$-polynomial gives the Tutte polynomial of the unoriented matroid underlying $N$. All three relations are stated in Theorem~\ref{thm:AisTutte}.\\

\fig{width=.5\linewidth}{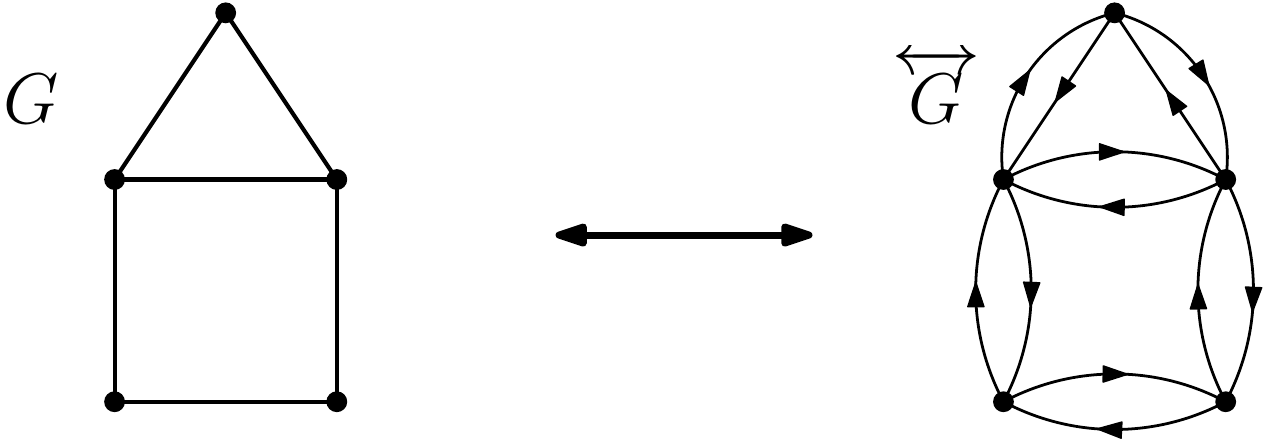}{Left: an unoriented graph. Right: the corresponding directed graph.}

Our goal is to investigate how much of the theory of the Tutte polynomial extends to the regular oriented matroid setting. Some highlights are the following.
\begin{itemize}
\item The invariant $\Apoly_N$ detects whether $N$ is acyclic. 
More generally, $\Apoly_N$ contains the generating function of acyclic submatroids of $N$ (counted according to their number of ground set elements) and the generating function of acyclic reorientations of $N$ (counted according to the number of reoriented elements).
\item Dually, the invariant $\Apoly_N$ detects whether $N$ is totally cyclic. 
More generally, $\Apoly_N$ contains the generating function of totally cyclic contractions of $N$, and the generating function of totally cyclic reorientations of $N$. 
\item  The $A$-polynomial can be used to define some invariants of \emph{partially oriented matroids} which generalize the Tutte polynomial in a satisfying manner.
\item There is a simple \emph{duality relation} between $\Apoly_N(-1,y,z)$ and $\Apoly_{N^*}(-1,y,z)$, where $N^*$ is the dual oriented matroid. \ob{There is an additional duality relation between $\Apoly_N(3,y,z)$ and $\Apoly_{N^*}(3,y,z)$.}
%
\end{itemize}


\ni \ob{\textbf{Closely and distantly related works}.}
In~\cite{Awan-Bernardi:B-poly} we defined an invariant of directed graphs, that we named \emph{$B$-polynomial}, which is also a generalization of the Tutte polynomial, but is not an oriented matroid invariant. Let us compare the definition of the $A$-polynomial and $B$-polynomial of a digraph. For a directed graph $D=(V,A)$ we denote by $\Apoly_D(q,y,z)$ the $A$-polynomial of the oriented matroid $N=N_D$ corresponding to $D$ (in the sense that the circuits of $N_D$ correspond to the simple cycles of $D$).
It is not hard to show that for all odd positive integers~$q=2p+1$,
\begin{equation}\label{eq:def-A-vertices}
\Apoly_D(q,y,z)=\frac{1}{q^{\comp(D)}}\sum_{f:V\to \{1,2,\ldots,q\}}y^{\#(u,v)\in A,~f(v)-f(u)\in S_q}\,z^{\#(u,v)\in A,~f(u)-f(v)\in S_q},
\end{equation}
where $S_q=\{1,2,\ldots,p\}\cup\{-p-1,-p-2,\ldots,-2p\}$, $\comp(D)$ is the number of connected components of $D$ and the sum is over the set of functions from the vertex set $V$ to $\{1,2,\ldots,q\}$. More details are given in Section~\ref{sec:B-link}.
For comparison, the $B$-polynomial of $D$ 
satisfies
\begin{equation*}\label{eq:defB}
\Bpoly_D(q,y,z)=\sum_{f:V\to \{1,2,\ldots,q\}}y^{\#(u,v)\in A,~f(u)>f(v)}\,z^{\#(u,v)\in A,~f(u)<f(v)},
\end{equation*}
for all positive integer~$q$. \\


The invariants $\Bpoly_D$ and $q^{\comp(D)}\,\Apoly_D$ are part of a family of digraph invariants defined in~\cite{Awan-Bernardi:B-poly}, all of which generalize the Tutte polynomial of graphs. In~\cite{Awan-Bernardi:B-poly} we extensively studied the $B$-polynomial, while the present article focuses on the $A$-polynomial. It must be noted that the $B$-polynomial is not an oriented matroid invariant: there are directed graphs with the same underlying oriented matroid, but different $B$-polynomials\footnote{\ob{In fact, digraphs with the same underlying oriented matroid very rarely have the same $B$-polynomials. For instance, different oriented trees tend to have different $B$-polynomials.}}. In particular, the $B$-polynomial cannot be extended to regular oriented matroids. In contrast, the $A$-polynomial is an oriented matroid invariant (as seen from its definition), and is well defined for all regular oriented matroids. While the $A$-polynomial and $B$-polynomial have somewhat different properties, there is still a simple relation between the specializations $\Apoly_D(-1,y,z)$ and $\Bpoly_D(-1,y,z)$ as shown in Proposition~\ref{prop:A=B-at-1}. Additionally, as we will explain in Section \ref{subsec:cyclically-sym}, for any directed graph $D$ there exists a ``cyclically-symmetric function'' refinement of $\Apoly_D$ which can be specialized to either $\Apoly_D$ or $\Bpoly_D$. \ob{In fact, we show that this cyclically-symmetric refinement of the $A$-polynomial contains the quasisymmetric $B$-polynomial studied in~\cite[Section 8]{Awan-Bernardi:B-poly} which generalizes to digraphs the Tutte symmetric function defined by Stanley for graphs~\cite{Stanley:symmetric-chromatic}, and the chromatic quasisymmetric function defined by Shareshian and Wachs for acyclic digraphs~\cite{Shareshian-Wachs:chromatic-quasisym-functions}.
We refer to~\cite{Awan-Bernardi:B-poly} for references concerning digraph invariants related to the $B$-polynomial and its quasisymmetric refinement.}\\

\ob{
Let us finally mention some other works that deal with notions of Tutte-polynomials for directed graphs and oriented matroids. The general framework of~\cite{KMT:Tutte-poly-framework} gives a canonical way to associate a ``Tutte polynomial'' to a class of combinatorial objects that have some notions of deletion and contraction which are compatible in an algebraic sense (based on Hopf algebras). However, for digraphs (or oriented matroids), the classical notions of deletions and contraction coincide with that of the underlying graph (or matroid), so the canonical Tutte polynomial given by this framework is simply the Tutte polynomial of the underlying graph (or matroid). We do not know of an alternative notion of deletion and contraction for digraphs which would give rise to the $A$-polynomial (nor to another generalization of the Tutte polynomial coinciding with the classical Tutte polynomial for unoriented graphs in the sense described above). 
Looking in a different direction, one can think of the Tutte polynomial  of graphs (originally called \emph{dichromate}) as the ``least common generalization''  of the flow polynomial (counting nowhere-zero $q$-flows) and coflow polynomial (counting nowhere-zero $q$-coflows), and use this as a guide to define a Tutte polynomial for generalizations of graphs. For instance, considering signed graphs (for which there is a natural generalization of nowhere-zero $q$-flows and $q$-coflows) this ``least common generalization'' approach leads to the Tutte polynomial of signed graphs of~\cite{GLRV:TuttePolySigned}. 
For digraphs one can consider some new notion of ``proper'' flows and coflows, and the corresponding polynomial invariant counting those. For instance, the notion of proper digraph coloring introduced by Neumann-Lara~\cite{Neumann-Lara:chrom-number-digraph}  leads to a notion of proper NL-flows (contracting the non-zero arcs gives a totally cyclic digraph) and proper NL-coflows (deleting the non-zero arcs gives an acyclic digraph) which are enumerated by polynomials~\cite{BHW:NL-flow-poly,HW:NL-coflow-poly}, and there exists a trivariate polynomial invariant of matroids which is the least common generalization of the NL-flow and NL-coflow polynomials~\cite{HW:NL-dichromate}. 
As for the $A$-polynomial, it is distantly related to a notion of proper $B$-flow introduced by Jaeger in~\cite{Jaeger:Bflows} (flows taking value in a set $B$), and further studied in~\cite{HN:Bflows,NR:Bflows}. Namely, the $A$-polynomial counts $\ZZ/q\ZZ$-coflows according to the number of arcs with value in the sets $\{0\}$, $B=\{1,2,\ldots,(q-1)/2\}$ and $-B=\{-1,-2,\ldots,-(q-1)/2$ (for odd $q$). While the $A$-polynomial can clearly be evaluated to count $B$-coflows (coflows with every value in $B$),  we do not know of whether the $A$-polynomial can be evaluated to count $B$-flows (flows with every value in $B$), nor whether there exists a ``least common generalization'' of the $B$-flow polynomial and $B$-coflow polynomials for regular matroids.\\
}

\ni \textbf{Outline.}
The paper is organized as follows. In Section~\ref{Background} we set up our notation and recall basic facts about oriented matroids. In Section~\ref{DefA}, we define and establish basic properties of the $A$-polynomial. We also establish the expression~\eqref{eq:def-A-vertices} for digraphs, and give the precise relation between the $A$-polynomial of digraphs and the family of digraph invariants defined in~\cite{Awan-Bernardi:B-poly}.
In Section~\ref{Potts_links}, we prove the three above-mentioned relations between the $A$-polynomial and the Tutte polynomial. 
In Section~\ref{CharPoly}, we define the \emph{weak-characteristic polynomial} and the \emph{strict-characteristic polynomial} of a regular oriented matroid. These invariants are related to the characteristic polynomial of unoriented matroids. We give several expansions of the $A$-polynomial in terms of weak- and strict-characteristic polynomials.
In Section~\ref{Ehrhart1} we finally prove the existence of the $A$-polynomial via Ehrhart theory. 
In Section~\ref{Ehrhart2} we use Ehrhart-Macdonald reciprocity in order to establish combinatorial interpretations of the evaluation of $\Apoly_N(q,y,z)$ at negative integer values of $q$. 
In Section~\ref{TuttePolynomials}, we recast and complement some of our results in the context of \emph{partially oriented matroids}, where the relations to the Tutte polynomial can be written in a more transparent manner. 
In Section~\ref{Q=0}, we study certain equivalence classes of the cocircuit reversing system for regular oriented matroids. We show that these equivalence classes are related to the polytopes appearing in the Ehrhart-theory underlying the $A$-polynomials and are counted by some evaluations of the \emph{characteristic quasi-polynomials}.
We conclude in Section~\ref{Perspectives} with some additional extensions and refinements of the $A$-polynomial\ob{, and many open questions}.

\section{Notation and Basic Facts about Matroids and Oriented Matroids}\label{Background}

\subsection{Basic notation}
For a positive integer~$n$, we denote by $[n]$ the set $\{1,\ldots,n\}$. The cardinality of a set $S$ is denoted by $|S|$ or $\# S$. We denote by $2^S$ the power set of $S$ and by $\RR^S$ the set of tuples indexed by $S$ with coordinates in $\RR$ (and we define similarly the sets $\ZZ^S$, $[0,1]^S$, etc.). 
For sets $R,S,T$, we write $R\uplus S = T$ to indicate that $R\cup S=T$ and $R\cap S=\emptyset$. We write $R\triangle S$ to denote the \emph{symmetric difference} of $R$ and $S$, which is defined as $R\triangle S=(R\cup S)\setminus (R\cap S)$.

For a positive real number $x$, we denote by $\lfloor x\rfloor$ the integral part of $x$.
For a polynomial $P(X)$ in a variable $X$, we denote by $[X^k]P(X)$ the coefficient of $X^k$ in $P(X)$. For a condition $C$, the symbol $\ONE_C$ has value $1$ if the condition $C$ is true, and $0$ otherwise. 


\subsection{Graphs and digraphs}
Our \emph{graphs} are finite. We allow loops and multiple edges. For a graph $G$, we write $G=(V,E)$ to indicate that $V$ is the set of vertices and $E$ is the set of edges. For an edge $e\in E$, we write $e=\{u,v\}$ to indicate that $u$ and $v$ are the endpoints of $e$ (and we allow $u=v$). We denote by $\comp(G)$ the number of connected components of $G$. 

A \emph{directed graph}, or \emph{digraph} for short, is a graph in which every edge $e$ is oriented (the endpoints are ordered). The oriented edges are called \emph{arcs} of the digraph. For a digraph $D$, we write $D=(V,A)$ to indicate that $V$ is the set of vertices and $A$ is the set of arcs. For an arc $a\in A$, we write $a=(u,v)$ to indicate that the arc $a$ has origin $u$ and end $v$ (and we allow $u=v$).

The \emph{incidence matrix} of a digraph $D=(V,A)$ is the matrix $\incidence_D=(r_{v,e})_{v\in V,e\in E}$ whose rows and columns are indexed by $V$ and $E$ respectively, and whose entries are given by 
$$r_{v,e}=\ONE_{\textrm{$v$ is the end of $e$}}-\ONE_{\textrm{$v$ is the origin of $e$}}.$$

The \emph{graph underlying a digraph} $D=(V,A)$ is the graph $\underline{D}$ obtained by forgetting the orientation of the arcs (equivalently, replacing each arc $a=(u,v)$ by an edge $e=\{u,v\}$). Conversely, an \emph{orientation} of a graph $G$ is a digraph $D$ with underlying graph $G$.
For a graph $G=(V,E)$, we denote by $\orient{G}=(V,A)$ the digraph obtained by replacing each edge $\{u,v\}\in E$ by two opposite arcs $(u,v)$ and $(v,u)$ (loops are also replaced by 2 arcs) as illustrated in Figure~\ref{fig:graph-as-digraph}.
 This operation identifies the set of graphs with the subset of digraphs such that arc set can be partitioned in pairs of arcs with the same endpoints but opposite directions. 

\subsection{Matroids}
We first recall the definition of matroids in terms of circuits.
A \emph{matroid} $M=(E,\mscr{C})$ consists of a finite \emph{ground set} $E$ and a set of \emph{circuits} $\mscr{C}\subseteq 2^E$, which satisfy the following:
\begin{enumerate}
\item $\emptyset\not\in \mscr{C}$.
\item If $C_1,C_2\in \mscr{C}$ are such that $C_1\subseteq C_2$, then $C_1=C_2$.
\item If $C_1,C_2\in \mscr{C}$ are such that there exists $e\in C_1\cap C_2$, then there exists $C_3\in \mscr{C}$ such that $C_3\subseteq (C_1\cup C_2)\setminus \{e\}$.
\end{enumerate}
Let $S$ be a subset of the ground set $E$. The set $S$ is called \emph{dependent} if it contains a circuit, and \emph{independent} otherwise. A \emph{basis} for $M$ is an independent subset of $E$, which is maximal by inclusion. 
The \emph{rank} of $S$, denoted by $\rk(S)$, is the maximal number of elements in an independent subset of~$S$. The \emph{rank of $M$} is $\rk(M):=\rk(E)$.
 
A \emph{cocircuit} is a subset $C^*\subseteq E$ such that $C^*$ intersects every basis, and is minimal by inclusion.
 An element $e\in E$ is a \emph{loop} if it is not in any basis of $M$, and it is a \emph{coloop} if it is in every basis of~$M$.

Let $B$ be a basis of $M$. For an element $e\in E\setminus B$, there is a unique circuit contained in $B\cup \{e\}$. It is called the \emph{fundamental circuit} for $e$ with respect to $B$, and denoted by $C_{B,e}$. Similarly, for $e\in B$ the \emph{fundamental cocircuit} for $e$ with respect to $B$, denoted by $C^*_{B,e}$, is the unique cocircuit contained in $(E\setminus B)\cup\{e\}$.

Let $M=(E,\mscr{C})$ be a matroid, and let $e\in E$. We call $M_{\setminus e}=(E,\mscr{C}_{\setminus e})$ the matroid obtained by \emph{deleting} $e$, where $\mscr{C}_{\setminus e} = \{X\in \mscr{C}\mid e\not\in X\}$. We call $M_{/e}=(E,\mscr{C}_{/e})$ the matroid obtained by \emph{contracting} $e$, where $\mscr{C}_{/e} = \{X\setminus \{e\}\mid X\in \mscr{C}\}\setminus \{\emptyset\}$. Note that if $e\in E$ is a loop or coloop, then $M_{\setminus e}=M_{/e}$. The operations of deletion and contraction commute: $M_{/a/b}=M_{/b/a}$, $M_{\setminus a\setminus b}=M_{\setminus b\setminus a}$ and $M_{\setminus a/b}=M_{/b\setminus a}$. Hence, for disjoint sets $S,R\subset E$, we can define the matroid $M_{\setminus S/R}$ obtained from $M$ by deleting the elements in $S$ and contracting the elements in $R$. 

Let $M=(E,\mscr{C})$ be a matroid, and let $\mscr{C}^*$ be the set of cocircuits of $M$. Then $M^* = (E,\mscr{C}^*)$ is also a matroid, called the \emph{dual matroid} of $M$. 


A matroid $M=(E,\mscr{C})$ is said to be \emph{represented over a field $K$ by a matrix $R$} if $R$ is a matrix with coefficients in $K$ and columns indexed by $E$ such that $\mscr{C}$ is the set of non-empty subsets $S$ of $E$ indexing columns which are minimally linearly dependent (the set of columns indexed by $S$ is linearly dependent but any proper subset is independent). In this case, we say that $M$ is \emph{representable} over $K$ and that $R$ is a \emph{representation} of $M$.


\subsection{Oriented matroids} \label{subsec:oriented-matroid}
We give a brief review of oriented matroids, and refer to~\cite{Bjorner:oriented-matroids} for more information. Let $A$ be a set. A \emph{signed set from $A$} is an ordered pair $X=(X^+,X^-)$ of disjoint subsets $X^+\subseteq A$ and $X^- \subseteq A$. The \emph{support} of the signed set $X=(X^+,X^-)$ is $\underline{X}=X^+\cup X^-$. 
The \emph{opposite} of a signed set $X$, denoted by $-X$, is the signed set with $(-X)^+=X^-$ and $(-X)^-=X^+$. 

An \emph{oriented matroid} $N = (A,\mfrk{C})$ consists of a finite \emph{ground set} $A$, and a set $\mfrk{C}$ of signed sets from~$A$ called \emph{signed circuits} (or just \emph{circuits} for short) which satisfy the following:
\begin{compactenum}
\item $(\emptyset,\emptyset)\not \in \mfrk{C}$,
\item for all $X\in \mfrk{C}$, $-X\in \mfrk{C}$,
\item for all $X, Y \in \mfrk{C}$, if $\underline{X}\subseteq \underline {Y}$, then $X=Y$ or $X=-Y$,
\item for all $X, Y \in \mfrk{C}$ such that $X\neq -Y$, and for all $a\in X^+\cap Y^-$, there exists $Z\in \mfrk{C}$ such that $Z^+\subseteq (X^+\cup Y^-)\setminus \{a\}$ and $Z^-\subseteq (X^-\cup Y^+)\setminus \{a\}$.
\end{compactenum}

For an oriented matroid $N=(A,\mfrk{C})$, we call \emph{underlying matroid} the matroid 
$ \underline{N}=(A,\underline{\mfrk{C}})$ where 
$\underline{\mfrk{C}}=\{\underline{X}\mid X\in \mfrk{C}\}$. 
Conversely, an \emph{orientation} of a matroid $M$ is an oriented matroid $N$ having underlying matroid $M$. We call a matroid $M$ \emph{orientable} if there exists an orientation for $M$. Not all matroids are orientable, but all matroids which are representable over $\RR$ are orientable. 

Let $N=(A,\mfrk{C})$ be an oriented matroid. Then there exists a unique orientation $N^*=(A,\mfrk{C}^*)$ of the matroid $(\underline{N})^*$ such that for all $C\in \mfrk{C}$ and $D\in \mfrk{C}^*$, 
the sets $(C^+\cap D^+)\cup (C^-\cap D^-)$ and $(C^+\cap D^-)\cup (C^-\cap D^+)$ are either both empty (that is, $\underline{C}\cap \underline{D}=\emptyset$) or both non-empty (that is, there are elements oriented in the same way in $C$ and $D$, and elements oriented in opposite ways in $C$ and $D$)~\cite[Proposition~3.4.1]{Bjorner:oriented-matroids}.
The oriented matroid $N^*$ is called the \emph{dual} of $N$. 
A signed set $D\subset A$ is a \emph{signed cocircuit} of $N$ if it is a signed circuit of $N^*$.

Let $N=(A,\mfrk{C})$ be an oriented matroid, let $B$ be a basis of $N$, and let $a\in A\setminus B$. Let $\underline{C}$ be the fundamental circuit of $a$ with respect to $B$ in $\underline{N}$. In $N$, there are two circuits $C_1$ and $C_2$ such that $\underline{C_1}=\underline{C_2}=\underline{C}$. We call the \emph{fundamental circuit} of $a$ with respect to $B$, denoted by $C_{B,a}$, the circuit of $N$ such that $a\in C_{B,a}^+$ and $\underline{C_{B,a}}=\underline{C}$. We define the \emph{fundamental cocircuit} of $a$ with respect to $B$ in the same way, and denote it by $C_{B,a}^*$.

Deletion and contraction of elements in an oriented matroid is defined in the same way as for unoriented matroids. Precisely, for an oriented matroid $N=(A,\mfrk{C})$ and $a\in A$, the \emph{oriented matroid obtained by deleting} $a$ is $N_{\setminus a}=(A\setminus\{ a\},\fC_{\setminus a})$, where $\fC_{\setminus a}:=\{C\in \fC \mid a\notin \un C\}$. The \emph{oriented matroid obtained by contracting} $a$ is $N_{/ a}=(A\setminus\{ a\},\fC_{/ a})$, where $\fC_{/ a}:=\{C\setminus \{a\},~C\in \fC,~\un C\neq \{a\}\}$ and $C\setminus \{a\}:=(C^+\setminus \{a\},C^-\setminus \{a\})$.
Let $N=(A,\mfrk{C})$ be an oriented matroid, and let $S\subseteq A$. The \emph{oriented matroid obtained by reorienting} $S$ is ${}_{-S}N:=(A,{}_{-S}\mfrk{C})$, where ${}_{-S}\mfrk{C}:=\{{}_{-S}C\mid C\in \mfrk{C}\}$ and ${}_{-S}C:=((C^+\setminus S)\cup (C^-\cap S),(C^-\setminus S)\cup (C^+\cap S))$. Reorientation commutes with deletion and contraction, so for disjoint subsets $R,S,T\subseteq A$ we can define the matroid ${}_{-S}N_{\setminus T/R}$ obtained by reorienting $S$, deleting $T$ and contracting $R$. 

We call a circuit or cocircuit $C$ \emph{positive} if $C^-=\emptyset$. Any element of an oriented matroid belongs either to a positive circuit or a positive cocircuit, but not both~\cite[Corollary~3.4.6]{Bjorner:oriented-matroids}. We call an element \emph{cyclic} if it belongs to a positive circuit, and \emph{acyclic} if it belongs to a positive cocircuit. We call an oriented matroid \emph{acyclic} if every element is acyclic, and \emph{totally cyclic} if every element is cyclic. 
If for two elements $a,b$, the signed set $(\{a,b\},\emptyset)$ is a circuit, then we say that $a$ and $b$ are \emph{opposite elements}.

An oriented matroid $N=(A,\mfrk{C})$ is said to be \emph{represented} by an $\RR$-matrix $R$ whose columns are labeled by the elements of $A$ if the underlying matroid $\un N$ is represented by $R$, and for any signed circuit $C\in\mfrk{C}$ there exists real values $k_a>0$ for all $a\in \underline{C}$ such that 
$$\displaystyle \sum_{a\in C^+} k_a\,R_a-\sum_{a\in C^-}k_a\,R_a=0,$$
where $R_a$ is the column of $R$ labeled by $a$.
For any matrix $R$ over $\RR$, there exists a unique oriented matroid represented by $R$, and we denote it by $N(R)$. 

For a digraph $D$, we denote by $N_D$ the oriented matroid whose circuits correspond to the simple cycles of $D$. Note that $N_D=N(\incidence_D)$, where $\incidence_D$ is the incidence matrix of $D$.

\subsection{Regular matroids and oriented matroids}
A matrix $R$ over $\RR$ is called \emph{totally unimodular} if every sub-determinant of $R$ has value in $\{-1,0,1\}$. 
A matroid or oriented matroid is called \emph{regular} if it is representable (over $\RR$) by a totally unimodular matrix. Regular matroids are also known as \emph{unimodular matroids} or \emph{signable matroids}.
We will list below a few useful facts about regular matroids. The interested reader can refer to~\cite{Oxley:matroid-theory,Bjorner:oriented-matroids} or~\cite[~Chapter 3]{White:combinatorial-geometries} for more details.

It is well known that the incidence matrix of any digraph $D$ is totally unimodular. Hence the \emph{graphic} matroids and oriented matroids (corresponding to graphs and digraphs respectively) are regular. An unoriented matroid is called \emph{binary} if it is representable over $\ZZ/2\ZZ$. The following are equivalent characterizations of regular matroids~\cite{Oxley:matroid-theory}. 
\begin{prop} The following are equivalent properties for a matroid:
\begin{compactenum}
\item $M$ is regular,
\item $M$ is representable over every field,
\item $M$ is binary and representable over some field of characteristic other than $2$,
\item $M$ is binary and orientable.
\end{compactenum}
\end{prop}







We now recall some properties of regular oriented matroids.
\begin{lemma}[{\cite[Corollary~7.9.4]{Bjorner:oriented-matroids}}]\label{lem:Reorientations}
If $N$ is a regular oriented matroid, then all orientations of $\underline{N}$ are reorientations of $N$.
\end{lemma}
\begin{lemma}
If an oriented matroid $N$ is regular, then its dual $N^*$ is also regular.
\end{lemma}

Let $N=(A,\fC)$ be an oriented matroid.
The \emph{incidence vector} of a signed set $X$ from $A$, is the vector $(x_a)_{a\in A}\in \{-1,0,1\}^A$ defined by $x_a=\ONE_{a\in X^+}-\ONE_{a\in X^-}$ for all $a\in A$.
The \emph{circuit lattice} of $N$ is the subgroup of the group $\ZZ^A$ generated by the incidence vectors of the circuits $N$ (equivalently, the set of linear combinations of incident vectors of circuits, with integer coefficients). 

\begin{lemma}[\cite{minty1966axiomatic}]\label{lem:CircuitOrthogonal}
Let $N=(A,\fC)$ be a regular oriented matroid, let $X$ be a circuit,  and let $Y$ be a cocircuit. Then, the  corresponding signed incidence vectors $(x_a)_{a\in A}$ and $(y_a)_{a\in A}$  satisfy $\sum_{a\in A}x_a y_a=0$ (in other words, they are orthogonal for the usual inner product on $\RR^A$). 
\end{lemma}

Next, we give some alternative characterizations of regular oriented matroids.

\begin{prop}\label{prop:ROMClassification}
Let $N$ be an oriented matroid. The following are all equivalent:
\begin{enumerate}
\item $N$ is regular,
\item $\underline{N}$ is regular,
\item $\underline{N}$ is binary,
\item for any basis $B$, the incidence vectors of the fundamental circuits of $N$ with respect to the basis $B$ generate the circuit lattice of $N$.
\end{enumerate}
\end{prop}
\begin{proof}
To show (1) is equivalent to (2), notice that if $N$ is regular, clearly the underlying matroid $\underline{N}$ is regular. On the other hand, suppose that $\underline{N}$ is regular. Let $R$ be some totally unimodular representation for $\underline{N}$. Then $N(R)$ is a regular oriented matroid, which is some orientation of $\underline{N}$. Since all orientations of a regular matroid are equivalent via reorientation, $N$ is also regular. Property (3) is equivalent to (1) by~\cite[Proposition~7.9.3]{Bjorner:oriented-matroids}. 
Property (4) was shown to be equivalent to (1) in~\cite{minty1966axiomatic}.
\end{proof}

\begin{lemma}\label{lem:FCrep}
Let $N=(A,\mfrk C)$ be a regular oriented matroid. Let $B$ be a basis for $N$, and let $F$ be the matrix whose rows are the incidence vectors of the fundamental circuits of $N$ with respect to the cobasis $A\setminus B$. Then $F$ is a totally unimodular representation for $N^*$. 
\end{lemma}
\begin{proof}
Since $N$ is regular, $N^*$ is also regular.
Let $R$ be a totally unimodular representation for $N^*$. By~\cite[~Theorem 11]{dall2014polyhedral} the integer lattice generated by the rows of $R$ is equal to the integer lattice generated by the incidence vectors of the circuits of $N$. 
Since the fundamental circuits of $N$ with respect to $B$ generate all of the circuits of $N$, the vector spaces generated by the rows of $R$ and $F$ are equal.
Hence the vector spaces $\Null(R)=\{(x_a)_{a\in A}\in \RR^A\mid Rx=0\}$ and $\Null(F)=\{(x_a)_{a\in A}\in \RR^A\mid F x=0\}$ are equal. Thus, the oriented matroids represented by $R$ and $F$ are equal. Hence $F$ is a representation over $R$ for the regular matroid $N^*$. Note that $F$ is in echelon form with all entries in $\{0,1,-1\}$ (assuming that the entries corresponding to $B$ are the first $|B|$ columns of $F$). This implies that $F$ is totally unimodular by~\cite[Lemma 3.1.2]{White:combinatorial-geometries} (since $\underline N^*$ is a binary matroid). 
\end{proof}

Lastly, we set some notation about the orientation of regular matroids. Let $M=(E,\mscr C)$ be a regular matroid. We denote by $\orient{M}=(A,\mfrk{C})$ the oriented matroid obtained by choosing an orientation $\vec{M}=(E,\vec \mC)$ of $M$, and then adding an opposite element for every element $e\in E$ of $\vec{M}$ (by \emph{adding an opposite element} for $e$, we mean adding a copy of $e$, and then reorienting it).
This operation is well defined (that is, $\orient{M}$ does not depend on the chosen orientation, up to isomorphism) since by Lemma~\ref{lem:Reorientations} all the orientations of $M$ are reorientations of $\vec{M}$. 
 Further, this operation identifies the set of regular matroids with the subset of regular oriented matroids whose ground sets can be partitioned into pairs of opposite elements.

\begin{definition}\label{def:OrientM}
Let $M$ be a regular matroid. We denote by $\Orient(M)$ the set of oriented matroids obtained from $\orient{M}$ by deleting one element from each pair of opposite elements. 
\end{definition}

\begin{note}\label{rk:OrientM}
We can think of $\Orient(M)$ as ``the set of orientations'' of the regular matroid $M=(E,\mscr C)$. One nice feature of this definition is that it coincides with the usual notion of the set of orientations of a graph, and it satisfies $|\Orient(M)|=2^{|E|}$.
Let us now explain the relation between $\Orient(M)$ and the set of distinct orientations of $M$ in the ``matroid sense'' of Section~\ref{subsec:oriented-matroid} (that is, the set of distinct oriented matroids $N$ such that $\un N=M$). Let $\vec M=(E,\vec \mC)$ be an orientation of $M$ in the matroid sense. By Lemma~\ref{lem:Reorientations} the orientations of $M$ are all the possible reorientations of $\vec M$. However, not all reorientations are distinct. For instance, ${}_{-E}\vec M=\vec M$. In general, one can consider the set $\Stab(M)=\{S\subseteq E\mid {}_{-S}\vec M=\vec M\}$ and the set $\mO(M)=\{{}_{-S}\vec M \mid S\subseteq E\}$ of distinct orientations of $M$.
The elements of $(\ZZ/2\ZZ)^{E}$ encoding the subsets in $\Stab(M)$ forms a subgroup of the additive group $(\ZZ/2\ZZ)^{E}$, and for each reorientation $\vec M'$ of $\vec M$ there is $|\Stab(M)|$ sets $S\subseteq E$ such that $ {}_{-S}\vec M=\vec M'$. Indeed, there is an underlying action of the additive group $(\ZZ/2\ZZ)^{E}$ on the set $\mO(M)$, and $\Stab(M)$ corresponds to the stabilizer of this action.
In conclusion, there are $|\mO(M)|=2^{|E|}/|\Stab(M)|$ distinct orientations of $M$, and there is a $|\Stab(M)|$-to-1 correspondence between $\Orient(M)$ and $\mO(M)$ which sends each oriented matroid in $\Orient(M)$ to an isomorphic oriented matroid in $\mO(M)$.
\end{note}





\medskip

\section{Definition and basic properties of the $A$-polynomial}\label{DefA}
In this section, we define the $A$-polynomial, and establish some of its immediate properties. We also give an equivalent definition of the $A$-polynomial for oriented matroids corresponding to digraphs, and prove some deletion-contraction relations.

\subsection{Coflows and the $A$-polynomials}
We first recall the definition of a $q$-coflow given in the introduction.

\begin{defn}
Let $N=(A,\mfrk{C})$ be an oriented matroid. Let $q$ be a positive integer and let $f:A\rightarrow \ZZ/q\ZZ$. We say that $f$ is a \emph{$q$-coflow} if for every circuit $C=(C^-,C^+)\in\mfrk{C}$, 
\begin{equation}\label{eq:coflow}
\sum_{a\in C^+} f(a) - \sum_{a\in C^-} f(a) = 0\quad (\text{in } \ZZ/q\ZZ).
\end{equation}
We denote by $F_N(q)$ the set of all $q$-coflows for $N$.

For a $q$-coflow $f$, we define 
$$f_A^>=\left\{a\in A\mid f(a)\in \bigg\{1,2,\ldots,\left\lfloor\frac{q}{2}\right\rfloor\bigg\}\right\},$$ 
where for an element $x\in \ZZ/q\ZZ$ and a set of integers~$S$, we write $x\in S$ to mean that $x=s+q\ZZ$ for some $s\in S$.
Similarly, we define $\ds f_A^<=\left\{a\in A\mid -f(a)\in \bigg\{1,2,\ldots,\left\lfloor\frac{q}{2}\right\rfloor\bigg\}\right\}$, 
$\ds f_A^==\{a\in A\mid f(a)=0\}$, $\ds f_A^\geq=f_A^>\cup f_A^=$, $\ds f_A^\leq=f_A^<\cup f_A^=$, and $\ds f_A^{\neq}=\{a\in A\mid f(a)\neq 0\}$. 
\end{defn}

The $q$-coflows of a (graphical) oriented matroid are indicated in Figure~\ref{fig:exp-Apoly}.
\fig{width=.6\linewidth}{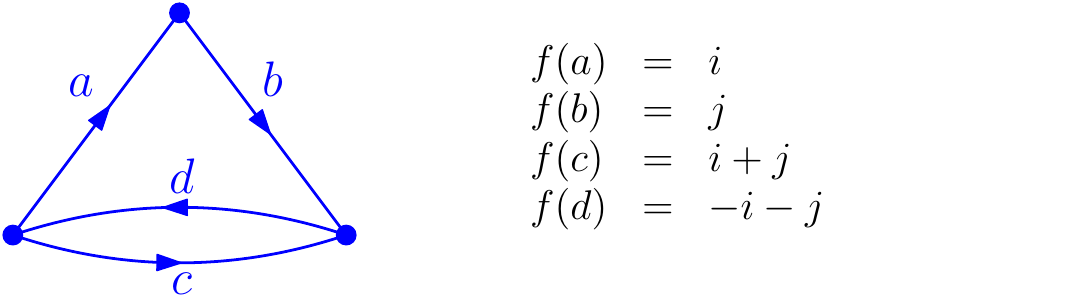}{A digraph $D$ with arc set $A=\{a,b,c,d\}$ and the associated coflows. Any $q$-coflow of the oriented matroid $N_D$ associated with $D$ has the form $f$ indicated on the right, where $i,j$ are arbitrary elements in $\ZZ/q\ZZ$.}

\begin{thm}\label{thm:Aqloring}
Let $N=(A,\mfrk{C})$ be a regular oriented matroid. There exists a unique trivariate polynomial $\Apoly_N(q,y,z)$ such that for all odd positive integers~$q$,
\[\Apoly_{N}(q,y,z)=\sum_{f\in F_N(q)} y^{|f_A^>|}z^{|f_A^<|}.\]
We call $\Apoly_N(q,y,z)$ the \emph{$A$-polynomial} of $N$.
\end{thm}

\begin{example}\label{exp:Apoly}
For the digraph $D$ of Figure~\ref{fig:exp-Apoly}, the oriented matroid $N=N_D$ has $A$-polynomial 
$$\Apoly_N(q,y,z)=yz\left(1+\frac{q-1}{2}(y+z)\right)^2 + (1-yz)\left(1+(q-1)yz\right).$$
For $q=3$, we get $\Apoly_N(3,y,z)=1+2yz+2yz^2+2y^2z+yz^3+y^3z$, which indicates the respective contributions of the nine $3$-coflows of $N$.
\end{example}
We postpone the proof of Theorem~\ref{thm:Aqloring} until Section~\ref{Ehrhart1}, and start by establishing some easy properties.
\begin{prop}\label{prop:Aeasy}
For any regular oriented matroid $N=(A,\mfrk{C})$ the following holds.
\begin{enumerate}[(i)]
\item $\Apoly_N(q,y,z) = \Apoly_N(q,z,y)$.
\item $\displaystyle\sum_{f\in F_N(q)}x^{|f_A^=|}y^{|f_A^>|}z^{|f_A^<|} = x^{|A|}\Apoly_N\left(q,\frac{y}{x},\frac{z}{x}\right)$.
\item $\displaystyle\sum_{f\in F_N(q)} y^{|f_A^\geq|}z^{|f_A^\leq|} = (yz)^{|A|}\Apoly_N\left(q,\frac{1}{y},\frac{1}{z}\right)$.
\item If $A=\emptyset$, then $\Apoly_N(q,y,z)=1$.
\item If $a \in A$ is a loop, then $\Apoly_N(q,y,z)=\Apoly_{N_{\setminus a}}(q,y,z)$.
\item If $a\in A$ is a coloop, then $\Apoly_N(q,y,z) = \left(1+\left(\frac{q-1}{2}\right)(y+z)\right)\Apoly_{N_{/a}}(q,y,z)$.
\item If $N$ is the disjoint union of two regular oriented matroids $N_1$ and $N_2$, then $\Apoly_N(q,y,z)=\Apoly_{N_1}(q,y,z)\Apoly_{N_2}(q,y,z)$.
\end{enumerate}
\end{prop}
\begin{proof}
Property (i) is obtained by the involution on $q$-coflows which changes the value $f(a)\in \ZZ/q\ZZ$ to $-f(a)$. Property (ii) follows from the fact that $|f^=_A|=|A|-|f_A^>|-|f_A^<|$. 
Property (iii) follows from (i) and (ii).
Property (iv) follows from the fact that there is a unique map $f:\emptyset\rightarrow S$ for any set $S$, hence a unique coflow for the empty matroid. 
Property (vii) follows from the fact that $f$ is a $q$-coflow for $N=N_1\cup N_2$ if and only if 
the restriction of $f$ to the ground sets $A_1$ of $N_1$ is a $q$-coflow for $N_1$ and its restriction to the ground sets $A_2$ of $N_2$ is a $q$-coflow for $N_2$, which gives 
$$\Apoly_{N}(q,y,z)=\sum_{f\in F_N(q)} y^{|f_A^>|}z^{|f_A^<|}=\!\sum_{f_1\in F_{N_1}(q),~f_2\in F_{N_1}(q)}\!\!\!\!\!\!\!\! y^{|{f_1}_{A_1}^>|+|{f_2}_{A_2}^>|}z^{|{f_1}_{A_1}^<|+|{f_2}_{A_2}^<|}=\Apoly_{N_1}(q,y,z)\Apoly_{N_2}(q,y,z).$$
Properties (v) and (vi) are special cases of Property (vii) corresponding to $N_2$ being either a loop (coflows have value 0 on loops) or a coloop (coflows are unconstrained on coloops).
\end{proof}

\smallskip

\subsection{The $A$-polynomial of digraphs}\label{sec:B-link}
In this section we consider the $A$-polynomial of digraphs. In this restricted setting we can give an alternative expression for the $A$-polynomial, which makes the connection with the digraph invariants considered in~\cite{Awan-Bernardi:B-poly} more apparent.

Let $D=(V,A)$ be a digraph, and let $N_D=(A,\fC)$ be the associated oriented matroid whose circuits correspond to the cycles of $D$. The \emph{$A$-polynomial of the digraph $D$} is defined as $\Apoly_{N_D}(q,y,z)$. 
A \emph{$q$-coloring} of $D$ is any function from $V$ to $[q]$. Let us recall the classical relation between the $q$-colorings of $D$ and the $q$-coflows of $N_D$ (see for instance~\cite{Diestel:graph-theory}). 

\begin{lemma}[Classical]\label{lem:ColoringToCoflow}
Let $D=(V,A)$ be a digraph, and let $q$ be a positive integer. Let $C_D(q)=\{f:V\rightarrow [q]\}$ be the set of $q$-colorings of $D$.
Let $\Phi:C_D(q)\rightarrow F_{N_D}(q)$ be the map given by $\Phi(f)(a)=f(v)-f(u)+q\ZZ$ for all $q$-colorings $f$ and all arcs $a=(u,v)\in A$. Then $\Phi$ is a $q^{\comp(D)}$-to-1 map, where $\comp(D)$ is the number of connected components of $D$.
\end{lemma}

\begin{proof}
We sketch the proof of Lemma~\ref{lem:ColoringToCoflow} for the reader's convenience.
First we claim that for any $q$-coloring $f\in C_D(q)$, the image $\Phi(f)$ is in $F_{N_D}(q)$. Indeed, for any circuit $C$ of $N_D$ the set of arcs of $D$ obtained from the arcs in $C^+$ together with the opposite of the arcs in $C^{-}$ gives a directed cycle $\vec{C}$ of the graph underlying $D$, hence
$$\displaystyle\sum_{a\in C^+}\Phi(f)(a)-\sum_{a\in C^-}\Phi(f)(a)=\sum_{(u,v)\in \vec{C}}f(u)-f(v)=0,$$
where the second sum is over the arcs of $\vec{C}$.

Now, for a coflow $g\in F_{N_D}(q)$ we consider the colorings $f:V\rightarrow[q]$ such that $\Phi(f)=g$. Notice that one can choose the value of $f$ on one vertex per connected component of $D$ and after this there is a unique $f$ such that $\Phi(f)=g$ (indeed Condition~\eqref{eq:coflow} ensures that the value of $f$ can be propagated without conflict throughout each component; we leave the details to the reader). 
\end{proof}

Using Lemma~\ref{lem:ColoringToCoflow}, one easily derives the following characterization of the $A$-polynomial of a digraph $D$ in terms of its $q$-colorings.
\begin{prop}\label{prop:AasB}
Let $D=(V,A)$ be a digraph. Then for any odd positive integer~$q$,
\begin{equation}\label{eq:AasB}
\Apoly_{D}(q,y,z) = \frac{1}{q^{\comp(D)}}\sum_{f:V\rightarrow [q]} y^{\#\{(u,v)\in A\mid f(v)-f(u)\in S_q\}}z^{\#\{(u,v)\in A\mid f(u)-f(v)\in S_q\}},
\end{equation}
where $S_q= \left\{1,2,\ldots,\left\lfloor \frac{q}{2}\right\rfloor\right\}\cup\left\{-q+1,-q+2,\ldots,-q+\left\lfloor\frac{q}{2}\right\rfloor\right\}$.
\end{prop}

Clearly,~\eqref{eq:AasB} uniquely determines the polynomial $\Apoly_D(q,y,z)$. Comparing this characterization to the definitions in~\cite[Section 9]{Awan-Bernardi:B-poly} immediately implies the following relation between $\Apoly_D$ and the invariant denoted by $\Bpoly^{(1,-1)}_D$ in~\cite{Awan-Bernardi:B-poly}:
\[q^{\comp(D)}\Apoly_{D}(q,y,z) = \Bpoly^{(1,-1)}_D(q,y,z).\]
The invariant $\Bpoly^{(1,-1)}_D$ is part of a family of digraph invariants $\Bpoly^w_D$, indexed by tuples $w$ with entries in $\{1,-1\}$ defined in~\cite[Section 9]{Awan-Bernardi:B-poly}. The $A$-polynomial is a regular oriented matroid extension of the digraph invariant  $\Bpoly^{(1,-1)}_D$. 

\ob{
\begin{remark}
It would be possible to define a regular oriented matroid extension for each of the digraph invariants  $\Bpoly^w_D$ for which the binary word $w$ is antipalyndromic. Indeed, as explained in~\cite[Section 9]{Awan-Bernardi:B-poly}, those digraph invariants can be defined in terms of coflows similarly as for $\Bpoly^{(1,-1)}_D$. 
\end{remark}
}




\smallskip


\subsection{Deletion-contraction relations}
In this subsection we establish some deletion-contraction-reorientation relations for the $A$-polynomial. 

There is a simple relation between the $A$-polynomial of an oriented matroid $N$ and those of the matroids $N_{\setminus a}$, $N_{/a}$, and ${}_{-a}N$ obtained by deleting, contracting and reorienting an element $a$.
\begin{lemma}\label{lem:Recurrence2}
Let $N=(A,\mfrk{C})$ be a regular oriented matroid, and let $a\in A$. If $a$ is not a coloop then,
\begin{equation}\label{GeneralRecurrence}
\Apoly_N(q,y,z)+\Apoly_{{}_{-a}N}(q,y,z) = (y+z)\,\Apoly_{N_{\setminus a}}(q,y,z)+(2-y-z)\,\Apoly_{N_{/a}}(q,y,z).
\end{equation}
If $a$ is a coloop, then $\Apoly_N(q,y,z) = \left(1+\left(\frac{q-1}{2}\right)y+\left(\frac{q-1}{2}\right)z\right)\Apoly_{N_{/a}}(q,y,z)$
\end{lemma}
\begin{proof}
The case of coloops is given by Proposition~\ref{prop:Aeasy}(vi). We now assume that $a$ is not a coloop.
We first observe that the coflows of $N$ and ${}_{-a}N$ are in bijection. Precisely, the bijection $\phi:F_{N}(q)\to F_{{}_{-a}N}(q)$ is given by $\phi(f)=\ov f$, where $\ov f(b)=f(b)$ for all $b\neq a$ and $\ov f(a)=-f(a)$. Thus for all positive odd integers $q$,
$$\Apoly_N(q,y,z)+\Apoly_{{}_{-a}N}(q,y,z)=\sum_{f\in F_N(q)} y^{|f_A^>|}z^{|f_A^<|}+y^{|\ov f_A^>|}z^{|\ov f_A^<|}.$$
Since $a$ is not a coloop, the coflows of $N$ and $N_{\setminus a}$ are in bijection. Precisely, the bijection $\psi:F_{N}(q)\to F_{N_{\setminus a}}(q)$ is given by $\psi(f)=\wt f$, where $\wt f(b)=f(b)$ for all $b\neq a$. Moreover, for all $f\in F_{N}(q)$,
$$y^{|f_A^>|}z^{|f_A^<|}+y^{|\ov f_A^>|}z^{|\ov f_A^<|}=(y+z)y^{|\wt f_{A\setminus a}^>|}z^{|\wt f_{A\setminus a}^<|}+\ONE_{f(a)= 0}(2-y-z)y^{|\wt f_{A\setminus a}^>|}z^{|\wt f_{A\setminus a}^<|}.$$
Lastly, the same map $\psi$ is also a bijection between $\{f\in F_N(q) \mid f(a)=0\}$ and $F_{N_{/a}}(q)$, thus
\begin{eqnarray*}
\Apoly_N(q,y,z)+\Apoly_{{}_{-a}N}(q,y,z)&=&(y+z)\sum_{f\in F_N(q)}y^{|\wt f_{A\setminus a}^>|}z^{|\wt f_{A\setminus a}^<|}+(2-y-z)\sum_{f\in F_N(q),f(a)=0}y^{|\wt f_{A\setminus a}^>|}z^{|\wt f_{A\setminus a}^<|}\\
&=&(y+z)\Apoly_{N_{\setminus a}}(q,y,z)+(2-y-z)\,\Apoly_{N_{/a}}(q,y,z).
\end{eqnarray*}
\end{proof}

Note that the relation in Lemma~\ref{lem:Recurrence2} cannot be used to compute the $A$-polynomial of $N$ in general. Consequently it does not play a significant role in the theory of the $A$-polynomial.
The next relation is useful when the matroid has a pair of opposite elements. For the oriented matroids $N=\orient{M}$ that correspond to an unoriented matroids this relation does determine the $A$-polynomial and in fact coincides with the classical deletion-contraction relation of the Tutte polynomial. 

\begin{lemma}\label{lem:Recurrence1}
Let $N=(A,\mfrk{C})$ be a regular oriented matroid. Let $e=\{a,b\}\subseteq A$ with $a$ and $b$ opposite elements in $N$. If $e$ is not the support of a cocircuit of $N$ then,
\begin{equation}\label{SpecialRecurrence}
\Apoly_N(q,y,z) = yz\,\Apoly_{N_{\setminus e}}(q,y,z)+(1-yz)\,\Apoly_{N_{/e}}(q,y,z).
\end{equation}
 If $e$ is the support for a cocircuit, then $\Apoly_N(q,y,z) = (1+(q-1)yz)\Apoly_{N_{/ e}}(q,y,z)$.
\end{lemma}
\begin{proof}
We will prove the case where $e=\{a,b\}$ is not the support for a cocircuit. Let $q$ be an odd positive integer. Then for any $f\in F_N(q)$,
\[y^{|f_A^>|}z^{|f_A^<|} = \left(yz+(1-yz)\ONE_{f(a)=0}\right)\left(y^{|f_{A\setminus e}^>|}z^{|f_{A\setminus e}^<|}\right).\]
Thus,
\begin{align*}
\Apoly_N(q,y,z) &= yz\sum_{f\in F_N(q)}y^{|f_{A\setminus e}^>|}z^{|f_{A\setminus e}^<|}+(1-yz)\sum_{f\in F_N(q)}\ONE_{f(e)=0}y^{|f_{A\setminus e}^>|}z^{|f_{A\setminus e}^<|}\\
&= yz\Apoly_{N_{\setminus e}}(q,y,z)+(1-yz)\Apoly_{N_{/e}}(q,y,z),
\end{align*}
where the last equality holds since $f\in F_{N\setminus e}(q)$ extends to a unique map $f'\in F_N(q)$.
The case where $e$ is the support of a cocircuit can be treated similarly.
\end{proof}
We will revisit relations~\eqref{GeneralRecurrence}~\eqref{SpecialRecurrence} in Section~\ref{TuttePolynomials} when considering invariants of partially oriented matroids.


\medskip

\section{Relations between the $A$-polynomial and the Tutte polynomial}\label{Potts_links}
In this section we prove three relations between the $A$-polynomial and the Tutte polynomial. It is easier to write down these relations in terms of the following reparametrization of the Tutte polynomial which we call \emph{Potts polynomial}.

\begin{definition} For an unoriented matroid $M=(E,\mscr{C})$, we denote by $T_M(x,y)$ the \emph{Tutte polynomial} of $M$ defined by 
\begin{equation}\label{eq:defTutte}
T_M(x,y) \defeq \sum_{S\subseteq E} (x-1)^{\rk(E)-\rk(S)}(y-1)^{|S|-\rk(S)}.
\end{equation}
where $\rk(S)$ is the rank of the subset $S$.
We call the polynomial 
\begin{equation}\label{eq:Potts-Tutte}
\Potts_M(q,y) = y^{|E|}(\sfrac{1}{y}-1)^{\rk(E)}T_M\bigg(1+\frac{q}{\sfrac{1}{y}-1},\sfrac{1}{y}\bigg)
\end{equation}
the \emph{Potts polynomial} of $M$.
\end{definition}

The relation~\eqref{eq:Potts-Tutte} between the Potts polynomial and Tutte polynomial is clearly invertible, so these matroid invariants are equivalent up to a change of variables. As we now explain, the Potts polynomial can be interpreted as counting coflows according to their number of non-zero values.

\begin{lemma}\label{lem:PottsTutteTheorem}
Let $M=(E,\mscr{C})$ be a regular unoriented matroid. For any orientation $\vec{M}$ of $M$, and every positive integer~$q$,
 \begin{equation}\label{PottsDefinition}
\Potts_M(q,y) = \sum_{f\in F_{\vec{M}}(q)} y^{|f_E^{\neq}|},
\end{equation}
where the sum is over the $q$-coflows of $\vec M$ and, as before, $f_E^{\neq}=\{a\in E~|~f(a)\neq 0\}$.
\end{lemma}



The name of the invariant $\Potts_M$ comes from its connection to a statistical mechanics model on graphs known as the \emph{Potts model}. Indeed, for an unoriented graph $G=(V,E)$ and the associated matroid $M_G$ on $E$ (whose circuits are the simple cycles of $G$), Equation~\eqref{PottsDefinition} can be translated in terms of $q$-colorings of $G$ using Lemma~\ref{lem:ColoringToCoflow}, and one gets
$$q^{\comp(G)} \Potts_{M_G}(q,y) = \sum_{f:V\to [q]} y^{\#\left\{\{u,v\}\in E\,\mid\, f(v)\neq f(u)\right\}},$$
for every positive integer~$q$.
The right hand side of this equation is the partition function of the $q$-states Potts model on $G$. The relation~\eqref{eq:Potts-Tutte} between this partition function and the Tutte polynomial is classical in the case of graphs; see for instance~\cite{Welsh:PottsTutte}. \ob{This partition function, or the equivalent invariant 
$$\sum_{f:V\to [q]} y^{\#\left\{\{u,v\}\in E\,\mid\, f(v)= f(u)\right\}},$$
which counts all $q$-colorings according to the number of monochromatic edges, appears under various names in the literature, such as \emph{monochromial} or \emph{bad colouring polynomial}.} Let us finally mention that for regular matroids the Potts polynomial is equivalent to an invariant defined by Sokal in~\cite[Section 3]{Sokal-multivariate-Tutte} in terms of a totally unimodular representation of $M$. 

Before proving Lemma~\ref{lem:PottsTutteTheorem}, let us use it to prove the main results of this section, which are the relations between the $A$-polynomial to the Tutte polynomial.

\begin{thm}\label{thm:AisTutte}
For any regular matroid $M=(E,\mscr{C})$, 
\begin{equation}\label{eqOrient}
\Apoly_{\orient{M}}(q,y,z) = {\Potts}_{M}(q,yz),
\end{equation}
where $\orient{M}$ is the oriented matroid corresponding to $M$.
Moreover,
\begin{equation}\label{eqAverage}
\frac{1}{2^{|E|}}\sum_{\vec{M} \in \Orient(M)} \Apoly_{\vec{M}}\left(q,y,z\right) = {\Potts}_{M}\left(q,\frac{y+z}{2}\right),
\end{equation}
where the sum is over the set of orientations of $M$ as specified in Definition~\ref{def:OrientM}.
Lastly, for any regular oriented matroid $N=(A,\mscr{C})$, 
\begin{equation}\label{eqUnderlying}
\Apoly_N(q,y,y) = \Potts_{\underline{N}}(q,y),
\end{equation}
where $\underline{N}$ is the unoriented matroid underlying $N$.
\end{thm}

Equation~\eqref{eqOrient} shows that, when restricted to unoriented regular matroids, the $A$-polynomial and Tutte polynomial are equal up to a change of variables (and the variable $z$ of the $A$-polynomial becomes redundant in this restricted case). This shows that the $A$-polynomial is a genuine generalization of the Tutte polynomial to the regular oriented matroid setting. Note also that~\eqref{eqAverage} shows that the $A$-polynomial of a regular oriented matroid $N$ captures all the information contained in the Tutte polynomial of $\un N$, such as the number of bases, the number of independent sets, etc.

\begin{remark}\label{rk:OrientM2}
Identity~\eqref{eqAverage} can be written in two equivalent ways via Remark~\ref{rk:OrientM}. Let $\vec M_0$ be an orientation of $M$. Since all the orientations of $M$ are reorientations of $\vec M_0$, we can rewrite~\eqref{eqAverage} as
\begin{equation}\label{eqAverage2}
\frac{1}{2^{|E|}}\sum_{S\subseteq E} \Apoly_{{}_{-S}\vec{M}_0}(q,y,z) = {\Potts}_{M}\left(q,\frac{y+z}{2}\right).
\end{equation}
Now, not all the reorientations of $\vec M_0$ are distinct. Let $\mO(M)=\{{}_{-S}\vec{M}_0,~S\subseteq E\}$ be the set of distinct orientations of $M$. By  Remark~\ref{rk:OrientM}, each orientation $\vec M \in \mO(M)$ contributes the same number of times to the sum in~\eqref{eqAverage2}.
Hence, one can rewrite~\eqref{eqAverage2} as
\begin{equation*}
\frac{1}{|\mO(M)|}\sum_{\vec M\in \mO(M)} \Apoly_{\vec{M}}\left(q,y,z\right) = {\Potts}_{M}\left(q,\frac{y+z}{2}\right).
\end{equation*}
In words, the average of the $A$-polynomial over all the orientations of $M$ is equivalent to the Tutte polynomial of $M$. 
\end{remark}

\begin{proof}[Proof of Theorem~\ref{thm:AisTutte}]
Since all these identities are between polynomials, it suffices to prove them for the specializations of $q$ at odd positive integers (and indeterminates $y,z$).

The proof of~\eqref{eqUnderlying} is straightforward: by definition,
$$\Apoly_N(q,y,y) = \sum_{f\in F_N(q)}y^{|f_A^{\neq}|},$$
which is equal to $\Potts_{\underline{N}}(q,y)$ by Lemma~\ref{lem:PottsTutteTheorem}.

Next we prove~\eqref{eqOrient}. 
Let $\vec{M}$ be an orientation of $M$. We need to prove 
\[\sum_{f\in F_{\orient{M}}(q)}y^{|f_{\orient{E}}^>|}z^{|f_{\orient{E}}^<|} = \sum_{f\in F_{\vec{M}}(q)}(yz)^{|f_E^{\neq}|}.\]
Recall that $\orient{M}$ is obtained from $\vec M$ by adding an element $e'$ opposite to $e$ for each $e\in E$.
Clearly, there is a bijection between the set of coflows $F_{\vec{M}}(q)$, and $ F_{\orient{M}}(q)$ given by $f\mapsto f'$, where $f'(e)=f(e)$ and $f'(e')=-f(e)$, for all elements $e\in E$. Hence it suffices to prove that for all $q$-coflows $f\in F_{\vec{M}}(q)$,
\begin{equation}\label{EqOneQloring}
y^{|f_{\orient{E}}'^>|}z^{|f_{\orient{E}}'^<|} = (yz)^{|f_E^{\neq}|}.
\end{equation}
Let $e\in E$. If $f(e)\neq 0$, then one of the elements $e$ and $e'$ will be in $f_{\orient{E}}'^>$ and the other will be in $f_{\orient{E}}'^<$. Hence, in this situation, the contribution is a factor $yz$ on both sides of~\eqref{EqOneQloring}.
On the other hand, if $f(e)=0$, then $f'(e)=f'(e')=0$, hence the contribution is a factor of 1 on both sides of~\eqref{EqOneQloring}. This proves~\eqref{EqOneQloring} and completes the proof.

It remains to prove~\eqref{eqAverage}, or equivalently~\eqref{eqAverage2} for some orientation $\vec M_0=(E,\vec C)$ of $M$. 
For all $S\subseteq E$ there is a bijection $\phi_S:F_{\vec M_0}(q)\to F_{{}_{-S}\vec M_0}(q)$ given by $\phi_S(f)={}_{-S}f$ for all coflow $f$, where ${}_{-S}f(e)=f(e)$ for all $e\in E\setminus S$ and ${}_{-S}f(e)=-f(e)$ for all $e\in S$. Hence 
$$\sum_{S\subseteq E} \Apoly_{{}_{-S}\vec{M}_0}(q,y,z)= \sum_{f\in F_{\vec M_0}(q)}~\sum_{S\subseteq E} y^{|{}_{-S}f_E^>|} z^{|{}_{-S}f_{E}^<|}.$$
Moreover the inner sum is easily seen to be equal to $2^{|f_E^=|}(y+z)^{|f_E^{\neq}|}$, hence we get
$$\frac{1}{2^{|E|}}\sum_{S\subseteq E} \Apoly_{{}_{-S}\vec{M}_0}(q,y,z)=\sum_{f\in F_{\vec N_0}(q)}\left(\frac{y+z}{2}\right)^{|f_E^{\neq}|}.$$
This gives~\eqref{eqAverage2} by Lemma~\ref{lem:PottsTutteTheorem}.
\end{proof}

Let us now prove Lemma~\ref{lem:PottsTutteTheorem}. We first state a result about $q$-coflows of regular oriented matroids.

\begin{lemma}\label{lem:QCoflowBasis}
Let $N=(A,\mfrk{C})$ be a regular oriented matroid, and let $B$ be a basis for $N$. 
\begin{enumerate}[(a)]
\item A map $f:A\rightarrow \ZZ/q\ZZ$ is a $q$-coflow if and only if for all fundamental circuits $C\in \{C_{B,a},~a\in A\setminus B\}$,
$\displaystyle\sum_{a\in C^+} f(a) - \sum_{a\in C^-} f(a) = 0.$
\item If $f_B:B\rightarrow \ZZ/q\ZZ$ is any map, then there is a unique extension of $f_B$, $f:A\rightarrow \ZZ/q\ZZ$ such that for all $b\in B$, $f(b)=f_B(b)$ and $f$ is a $q$-coflow.
\end{enumerate}
\end{lemma}
\begin{proof}
Claim (a) follows from Proposition~\ref{prop:ROMClassification}, which tells us that the circuit lattice of $N$ is generated by the fundamental circuits $C_{B,a}$ for $a\in A\setminus B$. Indeed, this property implies that for any circuit $C\in \fC$, the signed sum  $f(C):=\displaystyle\sum_{a\in C^+} f(a) - \sum_{a\in C^-} f(a)$ for $C$ can be expressed as sums and differences of the signed sums $f(C_{B,a})$ corresponding to fundamental circuits.
Claim (b) follows from Claim (a) since for every element $a\in A\setminus B$ the coflow condition on $C_{B,a}$ is satisfied by a unique value of $f(a)$.
\end{proof}


\begin{proof}[Proof of Lemma~\ref{lem:PottsTutteTheorem}]
Let $\vec{M}$ be an orientation of $M$.
For any positive integer~$q$,
\begin{eqnarray*}
\sum_{f\in F_{\vec{M}}(q)} y^{|f^=_E|}&=& \sum_{f\in F_{\vec{M}}(q)} \prod_{e\in E}(1+(y-1)\ONE_{f(e)=0})
=\sum_{f\in F_{\vec{M}(q)}}\sum_{S\subseteq E}(y-1)^{|S|}\ONE_{S\subseteq f_E^=}\\
&=&\sum_{S\subseteq E}(y-1)^{|S|}\sum_{f\in F_{\vec{M}}(q)} \ONE_{S\subseteq f_E^=}
=\sum_{S\subseteq E}(y-1)^{|S|}\left|F_{\vec{M}_{/S}}(q)\right|,
\end{eqnarray*}
where the last equality holds since $\big|F_{\vec{M}_{/S}}(q)\big|$ is precisely the number of $q$-coflows $f$ of $\vec{M}$ such that $f$ is zero on $S$.
We know from Lemma~\ref{lem:QCoflowBasis}~(b) that $q$-coflows on $\vec{M}_{/S}$ are in bijection with arbitrary maps $B\rightarrow \ZZ/q\ZZ$ for any basis $B$ of $\vec{M}_{/S}$. Since the size of a basis $B$ of $M_{/S}$ is $\rk(E)-\rk(S)$, we have $\big|F_{\vec{M}_{/S}(q)}\big|=q^{\rk{E}-\rk(S)}$ such coflows.
This gives 
$$\sum_{f\in F_{\vec{M}}(q)} y^{|f^=_E|}=\sum_{S\subseteq E}(y-1)^{|S|}q^{\rk(E)-\rk(S)}=(y-1)^{\rk(E)}T_M\bigg(1+\frac{q}{y-1},y\bigg).$$
Combining this with the observation $\sum_{f\in F_{\vec{M}}(q)} y^{|f^{\neq}_E|}=y^{|E|} \sum_{f\in F_{\vec{M}}(q)} (\sfrac{1}{y})^{|f^=_E|}$ shows that the Potts polynomial defined by~\eqref{eq:Potts-Tutte} satisfies~\eqref{PottsDefinition}.
\end{proof}


We note that  Lemma~\ref{lem:PottsTutteTheorem} does not hold for non-regular matroids, as shown by the following example.
\begin{example}
Let $U_{2,4}=(E,\mscr{C})$ be the matroid with $4$ elements, where the circuits are all the subsets of $E$ of size $3$. We will show that~\eqref{PottsDefinition} does not hold for this matroid.
Let $\vec{U}_{2,4}$ be the orientation of $U_{2,4}$ represented (over $\RR$) by the matrix $R$ below whose columns are indexed by $E=\{a,b,c,d\}$. 
\begin{figure}[h!]
\[R=\begin{array}{c}
\begin{array}{rrrr}
a&b&c&\phantom{-}d
\end{array}\\
\left(\begin{array}{rrrr}
1&0&1&1\\
0&1&1&-1
\end{array}\right)
\end{array},
\qquad
C=\begin{array}{c}
\begin{array}{rrrr}
a&\phantom{-}b&\phantom{-}c&\phantom{-}d
\end{array}\\
\left(\begin{array}{rrrr}
1&1&-1&0\\
1&-1&0&-1\\
1&0&-1&-1\\
0&1&-1&1
\end{array}\right)
\end{array}.
\]
\caption{A matrix $R$ representing $\vec{U}_{2,4}$, and a matrix $C$ giving the incidence vectors of the circuits of $\vec{U}_{2,4}$.}\label{fig:U24}
\end{figure}

The incidence vectors of the circuits of $\vec{U}_{2,4}$ are given by the rows of the matrix $C$ in Figure~\ref{fig:U24} (and their opposite). By definition, any $q$-coflow $f$ represented as a column vector $v$ in $(\ZZ/q\ZZ)^E$ would have to satisfy $C\,v=0$ in $(\ZZ/q\ZZ)^E$. Since $\det(C)=1$, the only solution is $v=0$, 
hence the only $q$-coflow on $\vec{U}_{2,4}$ is the zero map (in fact it can be verified that for any orientation of $U_{2,4}$ the only $q$-coflow is the zero map).
We can separately compute $T_{U_{2,4}}(x,y) =x^2+2x+2y+y^2$, and observe 
$$\sum_{f\in F_{\vec{U_{2,4}}}(q)} y^{|f_E^{\neq}|}=1\neq \Potts_{U_{2,4}}(q,y)=q^2y^4-4q y^4+4q y^3+3 y^4-4 y^3+1.$$
\end{example}

Let us finally mention that we do not know of a universal interpretation of the evaluations of $\Apoly_N(q,y,z)$ at even positive integer values of $q$. From~\eqref{eqOrient} we get that for a regular oriented matroid of the form $N=\orient{M}$ one has 
\[\Apoly_N(q,y^2,z^2) = \sum_{f\in F_N(q)} y^{2\,|f_A^>|}\,z^{2\,|f_A^<|}\,(yz)^{|f_A^{=q/2}|},\]
for all positive integers $q$, where $f_A^{=q/2} = \{a\in A\mid f(a) = q/2\}$. However, this expression does not generalize to arbitrary regular oriented matroids.

\medskip
\section{Oriented Characteristic Polynomials and their Generating Functions}\label{CharPoly}
\ob{In this section we will establish several expressions for the $A$-polynomial in terms of univariate polynomial invariants of oriented matroids. These expressions will prove useful to establish further results about the $A$-polynomial and its specializations.}

Recall that the \emph{characteristic polynomial} of a matroid $M=(E,\mscr{C})$ is the polynomial $\chrom_M$ given by 
\[\chrom_M(q)=(-1)^{\rk(M)}T_M(1-q,0),\]
Hence $\chrom_M(q)$ is related to the Potts polynomial by $\chrom_M(q)=[y^{|E|}]\Potts_M(q,y)$. If $M$ is regular, Lemma~\ref{lem:PottsTutteTheorem} gives
\begin{equation}\label{eq:characteristic-at-q}
\chrom_{M}(q) = \left|\{f\in F_{\vec{M}}(q)\mid \forall a\in A, f(a)\neq 0\}\right|,
\end{equation}
for any orientation $\vec{M}$ of $M$ and any positive integer~$q$.
This can be viewed as an extension of the classical relation 
\[\chrom_{M(G)}(q)=q^{-\comp(G)}\chi_G(q)\]
between the chromatic polynomial $\chi_G(q)$ of a graph $G$ and the characteristic polynomial of the associated matroid $M_G$.


We will now define two related polynomials $\chrom_N^>$ and $\chrom_N^\geq$ for regular oriented matroids.
\begin{lemma}\label{lem:existence-chrom}
Let $N=(A,\mfrk{C})$ be a regular oriented matroid. There exist polynomials $\chrom_N^>(X)$ and $\chrom_N^\leq(X)$ such that for all odd positive integers $q$,
\begin{eqnarray*}
\chrom_N^>(q) &=&\left|\left\{f\in F_N(q)\mid \forall a\in A, f(a)\in \left\{1,2,\ldots,\left\lfloor \frac{q}{2}\right\rfloor\right\} \right\}\right|,\\
\chrom_N^\geq(q) &=& \left|\left\{f\in F_N(q)\mid \forall a\in A, f(a)\in \left\{0,1,\ldots,\left\lfloor \frac{q}{2}\right\rfloor\right\} \right\}\right|.
\end{eqnarray*}
We call $\chrom_N^>(X)$ the \emph{strict-characteristic polynomial} of $N$, and $\chrom_N^\geq (X)$ the \emph{weak-characteristic polynomial} of $N$. 
\end{lemma}

Lemma~\ref{lem:existence-chrom} will be proved in Section~\ref{Ehrhart1}. Note that $\chrom_N^>(q) = [y^{|A|}]\Apoly_N(q,y,1)$ and $\chrom_N^\geq(q) = \Apoly_N(q,0,1)$. We now relate these polynomials to the classical characteristic polynomial.



\begin{lemma}\label{lem:characteristic}
For any regular matroid $M$,
\[\chrom_M(q) = \sum_{\vec{M}\in \Orient(M)}\chrom_{\vec{M}}^>(q),\]
where the sum is over the set of orientations of $M$ as specified in Definition~\ref{def:OrientM}.
\end{lemma}
\begin{proof}
Choose an orientation $\vec{M}_0$ for $M$. Since this is an identity between polynomials, it suffices to prove it for every odd positive integer~$q$. Using~\eqref{eq:characteristic-at-q} for an odd positive integer~$q$, we get
\begin{align*}
\chrom_M(q) &= \left|\{f\in F_{\vec{M}_0}(q)\mid \forall a\in A, f(a)\neq 0\}\right|\\
&=\left|\biguplus_{S\subseteq A}\left\{f\in F_{{}_{-S}\vec{M}_0}(q)\mid \forall a\in A, f(a)\in\left\{1,\ldots,\left\lfloor \frac{q}{2}\right\rfloor \right\}\right\}\right|\\
&=\sum_{S\subseteq A} \chrom_{{}_{-S}\vec{M}_0}^>(q)= \sum_{\vec{M}\in \Orient(M)}\chrom_{\vec{M}}^>(q),
\end{align*}
where the first identity is obtained by letting $S$ be the set of elements $a\in A$ such that $f(a)$ is in $\{-1,-2,\ldots,-\lfloor\frac{q}{2}\rfloor\}$. 
\end{proof}

We now state the main result of this section, which shows that the $A$-polynomial contains several generating functions of the strict and weak characteristic polynomials.
\begin{thm}\label{thm:Expansions1}
For any regular oriented matroid $N=(A,\mfrk{C})$, 
\begin{align}
\sum_{R\uplus S\uplus T=A}y^{|S|}\,z^{|T|}\, q^{\rk(N)-\rk(N_{\setminus R})}\, \chrom^>_{{}_{-T}N_{\setminus R}}(q)
&=\Apoly_N(q,1+y,1+z),\label{Exp1}\\
\sum_{R\uplus S\uplus T=A} y^{|S|}\,z^{|T|}\,q^{\rk(N)-\rk(N_{\setminus R})}\,\chrom_{{}_{-T}N_{\setminus R}}^\geq(q) 
&=(1+y+z)^{|A|}\Apoly_N\left(q,\frac{1+y}{1+y+z},\frac{1+z}{1+y+z}\right),\label{Exp2}\\
\sum_{R\uplus S\uplus T=A} y^{|S|}\,z^{|T|}\,\chrom^>_{{}_{-T}N_{/ R}}(q)
&=\Apoly_N(q,y,z),\label{Exp3}\\
\sum_{R\uplus S\uplus T=A} y^{|S|}\,z^{|T|}\,\chrom_{{}_{-T}N_{/ R}}^\geq(q)
&=(1+y+z)^{|A|}\Apoly_N\left(q,\frac{y}{1+y+z},\frac{z}{1+y+z}\right),\label{Exp4}
\end{align}
where the sums are over all possible ways of partitioning the ground set $A$ in three subsets $R$, $S$, and $T$.
\end{thm}
\begin{proof} Since these are identities between polynomials, it suffices to prove them for every odd positive integer~$q$. 
Let us start with~\eqref{Exp1}. By definition of $\Apoly_N$,
\[\Apoly_N(q,y,z) = \sum_{f\in F_N(q)}~\prod_{a\in A} \left( 1+(y-1)\ONE_{a\in f_A^>}+(z-1)\ONE_{a\in f_A^<}\right).\]
By expanding the product, one gets
\begin{align*}
\Apoly_N(q,y,z)&=\sum_{f\in F_N(q)}~\sum_{R\uplus S\uplus T=A} (y-1)^{|S|}(z-1)^{|T|}\,\ONE_{S\subseteq f_A^>\text{ and } T\subseteq f_A^<}\\
&=\sum_{R\uplus S\uplus T=A } (y-1)^{|S|}(z-1)^{|T|} \sum_{f\in F_N(q)} \ONE_{S\subseteq f_A^>\text{ and } T\subseteq f_A^<}.
\end{align*}
In order to compute the inner sum, let us fix a partition $R\uplus S\uplus T=A$ and show that the sets 
$$F:=\{f\in F_N(q)\mid S\subseteq f_A^>\text{ and } T\subseteq f_A^<\}\textrm{ and }G:=\{g\in F_{N_{\setminus R}}(q) \mid S= g_A^>\textrm{ and }T= g_A^<\}$$ 
satisfy $|F|=q^{\rk(N)-\rk(N_{\setminus R})}|G|$. First observe that a coflow $f\in F_N(q)$ is in $F$ if and only if its restriction to $A\setminus R$ is in $G$. Moreover we claim that any $g\in G$ is the restriction of exactly $q^{\rk(N)-\rk(N_{\setminus R})}$ coflows in $F$. Indeed, consider a basis $B$ of $N_{\setminus R}$ and complete it into a basis $B'$ of $N$. By Lemma~\ref{lem:QCoflowBasis}(b), the maps in $G$ are uniquely determined by their values on $B$ and their extensions in $F$ are determined by their values on $B'\setminus B$. Since $|B'\setminus B|=\rk(N)-\rk(N_{\setminus R})$ there are $q^{\rk(N)-\rk(N_{\setminus R})}$ possible extensions for each $g\in G$. 
This proves our claim, hence
$$\sum_{f\in F_N(q)} \ONE_{S\subseteq f_A^>\text{ and } T\subseteq f_A^<}~=~|F|~=~q^{\rk(N)-\rk(N_{\setminus R})}\,|G|~=~ q^{\rk(N)-\rk(N_{\setminus R})}\,\chrom^>_{{}_{-T}N_{\setminus R}}(q),$$
which completes the proof of~\eqref{Exp1} (after replacing $y$ and $z$ by $1+y$ and $1+z$ respectively).

The proof of~\eqref{Exp2} is similar except one starts with the identity
\[\Apoly_N(q,y,z) = \sum_{f\in F_N(q)}~\prod_{a\in A} \left( y+z-1+(1-z)\ONE_{a\in f_A^\geq }+(1-y)\ONE_{a\in f_A^\leq}\right).\]

For the proof of~\eqref{Exp3} we start with the identity
 \[\Apoly_N(q,y,z) = \sum_{f\in F_N(q)}~\prod_{a\in A} \left(\ONE_{a\in f_A^=}+y\ONE_{a\in f_A^> }+z\ONE_{a\in f_A^<}\right),\]
which gives 
\begin{align*}
\Apoly_N(q,y,z)&=\sum_{f\in F_N(q)}~\sum_{R\uplus S\uplus T=A} y^{|S|}z^{|T|}\,\ONE_{R=f_A^=\text{ and } S= f_A^>\text{ and } T= f_A^<}\\
&=\sum_{R\uplus S\uplus T=A } y^{|S|}z^{|T|} \sum_{f\in F_N(q)}\ONE_{R=f_A^=\text{ and } S= f_A^>\text{ and } T= f_A^<}.
\end{align*}
Moreover, for any subset $R\subseteq A$, a map $f:A\to \ZZ/q\ZZ$ such that $f(a)=0$ for all $a\in R$ is a $q$-coflow of $N$ if and only if its restriction to $A\setminus R$ is a $q$-coflow of $N_{/R}$. Hence we get 
$$ \sum_{f\in F_N(q)}\ONE_{R=f_A^=\text{ and } S= f_A^>\text{ and } T= f_A^<}~=~ \# \{g\in F_{N_{/R}}(q) \mid S= g_A^>\text{ and } T= g_A^<\}~=~\chrom^>_{{}_{-T}N_{/ R}}(q),$$
which completes the proof of~\eqref{Exp3}.

The proof of~\eqref{Exp4} is similar except one starts with the identity
\[\Apoly_N(q,y,z) = \sum_{f\in F_N(q)}~\prod_{a\in A} \left((1-y-z)\ONE_{a\in f_A^=}+y\ONE_{a\in f_A^\geq }+z\ONE_{a\in f_A^\leq}\right).\]
\end{proof}

Each of the generating functions in Theorem~\ref{thm:Expansions1} involves two operations among deletion, contraction, and reorientation. As specializations, we can get generating functions involving a single operation. For deletion, we get
\[\sum_{S\subseteq A} y^{|A\setminus S|}\, q^{\rk(N)-\rk(N_{\setminus S})}\,\chrom_{N_{\setminus S}}^>(q) = \Apoly_N(q,y+1,1),\]
\[\sum_{S\subseteq A} y^{|A\setminus S|}\, q^{\rk(N)-\rk(N_{\setminus S})}\,\chrom_{N_{\setminus S}}^\geq (q) = (y+1)\Apoly_N\left(q,\frac{1}{y+1},1\right).\]
For contraction, we get
\[\sum_{S\subseteq A}y^{|A\setminus S|}\, \chrom_{N_{/S}}^>(q) = \Apoly_N(q,y,0),\]
\[\sum_{S\subseteq A}y^{|A\setminus S|} \,\chrom_{N_{/S}}^\geq (q) = (y+1)^{|A|}\Apoly_N\left(q,\frac{y}{y+1},0\right).\]
For reorientation, we get
\begin{equation}\label{eq:reorientGF1}
\sum_{S\subseteq A} \alpha^{|S|}\,\chrom_{{}_{-S}N}^>(q) = [y^{|A|}]\Apoly_N(q,y,\alpha y),
\end{equation}
\begin{equation}\label{eq:reorientGF2}
\sum_{S\subseteq A} \alpha^{|S|}\,\chrom_{{}_{-S}N}^\geq(q) = (1+\alpha)^{|A|}\Apoly_N\left(q,\frac{1}{1+\alpha},\frac{\alpha}{1+\alpha}\right).
\end{equation}
Let us explain how to obtain~\eqref{eq:reorientGF2}.
First note
$$\sum_{S\subseteq A} \alpha^{|S|}\,\chrom_{{}_{-S}N}^\geq(q)=[x^{|A|}]\sum_{R\uplus S\uplus T=A} x^{|S|}\,(\al x)^{|T|}\,\chrom_{{}_{-T}N_{/ R}}^\geq(q).$$
Using~\eqref{Exp4} gives
\begin{eqnarray*}
\sum_{S\subseteq A} \alpha^{|S|}\,\chrom_{{}_{-S}N}^\geq(q)&=&[x^{|A|}](1+x+\al x)^{|A|}\Apoly_N\left(q,\frac{x}{1+x+\al x},\frac{\al x}{1+x+\al x}\right)\\
&=&[x^{|A|}]\sum_{f\in F_N(q)}x^{|f_{A}^>|}(\al x)^{|f_{A}^<|}(1+x+\al x)^{|f_A^=|}\\
&=&\sum_{f\in F_N(q)}1^{|f_{A}^>|}\al^{|f_{A}^<|}(1+\al)^{|f_A^=|}=(1+\alpha)^{|A|}\Apoly_N\left(q,\frac{1}{1+\alpha},\frac{\alpha}{1+\alpha}\right).
\end{eqnarray*}
The proof of~\eqref{eq:reorientGF1} is similar.

\medskip
\section{Proof of the existence of the $A$-polynomial}\label{Ehrhart1}
The goal of this section is to prove the existence of the weak and strict-chromatic polynomials (Lemma~\ref{lem:existence-chrom}) and of the $A$-polynomial (Theorem~\ref{thm:Aqloring}). Our proofs are based on Ehrhart theory, and the key is to interpret the values of characteristic polynomials in terms of the number of lattice points in certain polytopes.

We first give a brief review of Ehrhart theory; see~\cite{Breuer2015} for a more detailed introduction. Recall that for a set $\Delta\subseteq \RR^n$, and a positive real number $q$, the \emph{$q$-dilation} of $\Delta$ is the set
\[q\Delta\defeq \{(q x_1,\ldots,q x_n)\mid (x_1,\ldots,x_n)\in \Delta\}\subseteq \RR^n.\]
A finite region $\Pi\subset \RR^n$ is a \emph{polytope} if it can be written $\Pi = \{x\in \RR^n\mid Ax\leq b\}$ for some matrix $A$ and vector $b$. 
Let $\Pi$ be a polytope. The \emph{dimension} of $\Pi$, denoted by $\dim(\Pi)$, is the dimension of the affine subspace spanned by $\Pi$.
The \emph{relative interior} of $\Pi$, denoted by $\Pi^\circ$, is the topological interior of $\Pi$ in the affine subspace spanned by $\Pi$ (with the subspace topology). 
A function $Q:\ZZ\to \RR$ is a \emph{quasipolynomial of period $p$} if there exists polynomials $P_1,\ldots,P_p\in \RR[X]$ such that for all $i\in [p]$, $Q(n)=P_i(n)$ for all integers~$n$ in $i+p\ZZ$. The main results of Ehrhart theory~\cite{ehrhart1962geometrie,macdonald1971polynomials} are stated below.

\begin{thm}[Ehrhart's Theorem and Ehrhart-Macdonald reciprocity] \label{thm:Ehrhart}
Let $\Pi\subset \RR^d$ be a polytope whose vertices have rational coordinates. Let $p$ be the least common multiple of the denominators of the coordinates of the vertices of $\Pi$.
Then there exists a unique quasipolynomial $E_\Pi$ of period $p$, called the \emph{Ehrhart quasipolynomial} of $\Pi$, such that
\[E_{\Pi}(q)=\left|q\Pi\cap \ZZ^d\right|,\]
for all non-negative integers~$q$. Moreover, for all positive integers~$q$, 
\begin{equation}\label{eq:Ehrhart-reciprocity}
(-1)^{\dim(\Pi)}E_{\Pi}(-q) = \left|q\Pi^\circ\cap \ZZ^d\right|.
\end{equation}
\end{thm}

\begin{exmp}
Consider the polytope $\Pi=\{(x,y)\in \RR^2\mid x,y\in [0,1/2]\}$. 
Let $P_1=\left(\frac{X+1}{2}\right)^2$ and $P_2=\left(\frac{X+2}{2}\right)^2$. It is easy to check that for all non-negative integer~$q$,
$\left|q\Pi\cap \ZZ^d\right|=P_1(q)$ if $q$ is odd and $\left|q\Pi\cap \ZZ^d\right|=P_2(q)$ if $q$ is even. Hence the Ehrhart quasipolynomial of $\Pi$ is equal to 
$$E_{\Pi}(q)=\left(\frac{q+1+\ONE_{\text{$q$ is even}}}{2}\right)^2.$$
for all $q$ in $\ZZ$. We can also check that the Ehrhart-Macdonald reciprocity~\eqref{eq:Ehrhart-reciprocity} holds:
$$\left|q\Pi^\circ \cap \ZZ^2\right|=(-1)^2\,E_{\Pi}(-q)=\left(\frac{-q+1+\ONE_{\text{$-q$ is even}}}{2}\right)^2=\left(\frac{q-1-\ONE_{\text{$q$ is even}}}{2}\right)^2,$$
for all positive integers~$q$.
\end{exmp}

We will now relate the strict and weak-characteristic polynomials to some Ehrhart quasipolynomials.
Let $N=(A,\mfrk{C})$ be an oriented matroid. To a $q$-coflow $f\in F_N(q)$ we associate the point $x=(x_a)_{a\in A}\in \{0,1,\ldots, q-1\}^A$ such that $f(a)=x_a +q\ZZ$ for all $a\in A$. This correspondence gives a bijection between $F_N(q)$ and the set of points $x\in \{0,1,\ldots, q-1\}^A$ such that for all circuits $C\in \mfrk{C}$, the sum $\sum_{a\in C^+}x_a-\sum_{a\in C^-}x_a$ is a multiple of $q$.
Through this bijection, we get
$$\chrom_N^\geq(q)
=\# \bigg\{x\in \ZZ^A\mid \forall a\in A, ~0\leq x_a\leq q/2,\textrm{ and } \forall C\in \mfrk{C},~\sum_{a\in C^+}x_a-\sum_{a\in C^-}x_a\in q\ZZ \bigg\},$$
for all odd positive integer~$q$.

Let $\Pi_N\subseteq \RR^{|A|}$ be the set of points $x=(x_a)_{a\in A}\in \RR^A$ such that 
\begin{equation*}
\forall a\in A,~  0\leq x_a\leq 1/2,
~\textrm{ and }~
\forall C\in \mfrk{C},~ \sum_{a\in C^+}x_a-\sum_{a\in C^-}x_a\in \ZZ.
\end{equation*}
By the preceding, for all odd positive integer $q$, $\chrom_N^\geq(q)$ is equal to the number of lattice points in the dilation $q \Pi_N$.

Next we observe also that $\Pi_N$ is a disjoint union of a finite number of polytopes, each defined in terms of the value of the sums $\sum_{a\in C^+}x_a-\sum_{a\in C^-}x_a$ for each circuit $C\in \mfrk{C}$. This is illustrated in the following example.

\begin{example}\label{ExampleCubeChi}
Let $N=(\{a,b,c,d\},\mfrk{C})$ be the regular oriented matroid such that $\mfrk{C}$ consists of the two circuits $C=(\{a,b,c\},\{d\})$ and $-C$. Then for any odd positive integer~$q$, $\chrom_{N}^\geq(q)=|q\Pi_{N}\cap \ZZ^4|$, where $\Pi_{N}$ is the set of points $x=(x_a,x_b,x_c,x_d)\in [0,\sfrac{1}{2}]^4$ such that $x_a+x_b+x_c-x_d\in \ZZ$. 
 The set $\Pi_{N}$ is the disjoint union of two polytopes: $\Pi_{N,0}$ consisting of the subset of points such that $x_a+x_b+x_c-x_d=0$, and $\Pi_{N,1}$ consisting of the subset of points such that $x_a+x_b+x_c-x_d=1$.
These polytopes have dimension 3 and their projections to $\RR^{\{a,b,c\}}$ are represented in Figure~\ref{FigureCubeChi}. 
\begin{figure}[h]
\includegraphics[scale=0.35]{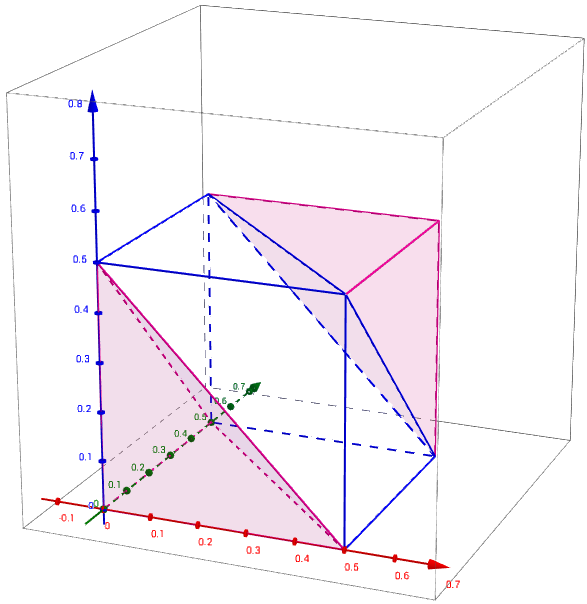}
\caption{The set $\Pi_{N}=\Pi_{N,0}\uplus \Pi_{N,1}$ for the matroid $N$ of Example~\ref{ExampleCubeChi}, projected to $\RR^{\{a,b,c\}}$. The projection of $\Pi_{N,0}$ consists of the points $(x_a,x_b,x_c)\in [0,\sfrac{1}{2}]^3$ such that $0\leq x_a+x_b+x_c\leq \sfrac{1}{2}$, while the projection of $\Pi_{N,1}$ consists of the points $(x_a,x_b,x_c)\in [0,\sfrac{1}{2}]^3$ such that $1\leq x_a+x_b+x_c\leq \sfrac{3}{2}$.
}
\label{FigureCubeChi}
\end{figure}
\end{example}

Let us set some additional notation. For a signed set $C=(C^+,C^-)$ of elements in $A$, we denote by $\ov{C}\in \{-1,0,1\}^A$ the signed incidence vector of $C$, that is, $\ov{C}=(y_a)_{a\in A}$ where $y_a=1$ if $a\in C^+$, $y_a=-1$ if $a\in C^-$, and $y_a=0$ otherwise. We denote by $\langle\_,\_\rangle$ the standard inner product on $\RR^A$, defined by $\langle x,y\rangle=\sum_{a\in A} x_a y_a$. Note that for a signed set $C\in \fC$ and $x=(x_a)\in \RR^A$ one has
$$\langle \ov C, x\rangle=\sum_{a\in C^+}x_a-\sum_{a\in C^-}x_a.$$
With this notation,
$$\Pi_N=\big\{x=(x_a)_{a\in A}\in \RR^A\mid\forall a\in A,~0\leq x_a\leq 1/2,\textrm{ and } \forall C\in \fC,~\langle \ov C, x\rangle\in \ZZ \big\}.$$

For $x\in \Pi_N$, let $\alpha_x:\mfrk{C}\to \ZZ$ be the function defined by $\alpha_x(C)=\langle \ov C, x\rangle$. Let 
$$\Om_N^\geq=\{\alpha_x,~x\in \Pi_N\},$$
and for $\alpha\in \Om_N^\geq$, let
$$\Pi_{N,\alpha}=\{x\in \Pi_N\mid \alpha_x=\alpha\}.$$
We have $\ds \Pi_N=\biguplus_{\alpha\in \Om_N^\geq} \Pi_{N,\alpha},$ so that 
\begin{equation}\label{eq:weak-chrom-Ehrhart}
\chrom_N^\geq(q)=|q\Pi_N\cap \ZZ^A|=\sum_{\al \in \Om_N^\geq}|q\Pi_{N,\al}\cap \ZZ^A|.
\end{equation}

Similarly, let
$$\Pi_N^>=\big\{x=(x_a)_{a\in A}\in \RR^A\mid\forall a\in A,~0< x_a< 1/2,\textrm{ and } \forall C\in \fC,~\langle \ov C, x\rangle\in \ZZ \big\}\subset \Pi_N.$$
Letting
$$\Om_N^>=\{\alpha_x,~x\in \Pi_N^>\},$$
we get $\ds \Pi_N^>=\biguplus_{\alpha\in \Om_N^>} \Pi^\circ_{N,\alpha},$ and 
\begin{equation}\label{eq:strict-chrom-Ehrhart}
\chrom_N^>(q)=|q\Pi_N^>\cap \ZZ^A|=\sum_{\al \in \Om_N^>}|q\Pi_{N,\al}^\circ\cap \ZZ^A|.
\end{equation}

In order to apply Ehrhart theory, we now examine the coordinates of the vertices of the polytopes~$\Pi_{N,\alpha}$.


\begin{lemma}\label{lem:HalfInteger}
Let $N=(A,\mfrk{C})$ be a regular oriented matroid. For all $\alpha\in \Omega_N^\geq$, the vertices of $\Pi_{N,\alpha}$ have half-integer coordinates. Moreover, for all $\alpha\in \Omega_N^>$, the dimension of $\Pi_{N,\alpha}$ is $\rk(N)$.
\end{lemma}
\begin{proof}
Let $\alpha\in \Omega_N^\geq$.
Let $B\subseteq A$ be a basis of $N$. 
For $a\in A\setminus B$, let $C_{B,a}\in \fC$ be the fundamental circuit of $a$ with respect to $B$. 
We claim that a point $x=(x_a)_{a\in A}\in \RR^A$ is in $\Pi_{N,\alpha}$ if and only if 
\begin{equation}\label{eq:reducedeq-for-Pi}
\forall a\in A,~ 0\leq x_a\leq \sfrac{1}{2},
~\textrm{ and }~
\forall a\in A\setminus B,~ \langle \ov C_{B,a},x\rangle=\alpha(C_{B,a}).
\end{equation}
Indeed, by Proposition~\ref{prop:ROMClassification}, the incidence vectors $\{\ov C_{B,a},~a\in A\setminus B\}$ generate the circuit lattice of $N$, hence a point $x$ satisfying~\eqref{eq:reducedeq-for-Pi} automatically satisfies $\langle \ov C, x\rangle=\alpha(C)$ for all $C\in \fC$. 

We first prove the statement about the dimension of $\Pi_{N,\al}$. Note that the equations $\langle \ov C_{B,a},x\rangle=\alpha(C_{B,a})$ for $a\in A\setminus B$ are linearly independent (since the equation $\langle \ov C_{B,a},x\rangle=\alpha(C_{B,a})$ is the only one involving the coordinate $x_a$), so the subspace $V=\{x\in \RR^A\mid\forall a\in A\setminus B,~ \langle\ov C_{B,a},x\rangle=\alpha(C_{B,a})\}$ has dimension $|A|-|A\setminus B|=|B|=\rk(N)$. For all $\alpha\in \Om_N^>$, the intersection of $V$ with the open set $(0,1/2)^A$ is non-empty, hence it has dimension $\rk(N)$. Thus, $\Pi_{N,\al}=V\cap[0,1/2]^A$ has dimension $\rk(N)$.

It remains to prove the statement about the vertices of $\Pi_{N,\alpha}$. We first express~\eqref{eq:reducedeq-for-Pi} in matrix form.
Let $F$ be the matrix whose columns are indexed by $A$ and whose rows are the signed incidence vectors for the fundamental circuits $\{C_{B,a}\}_{a\in A\setminus B}$. 
Let $\overline{\alpha}^B$ be the vector $(\alpha_a^B)_{a\in A\setminus B}$, where $\alpha^B_a=\alpha(C_{B,a})$ for all $a\in A\setminus B$. Let $n=|A|$ and let
\[P=\left[\begin{array}{r} Id_{A}\\-Id_{A}\\F\\-F \end{array}\right]\quad \text{and}\quad c=\left[\begin{array}{c}(\sfrac{1}{2})_{n}\\(0)_{n}\\\overline{\alpha}^B\\-\overline{\alpha}^B\end{array}\right],\]
where $P$ has columns indexed by $A$, $\Id_A$ is the $n\times n$ identity matrix indexed by $A$,  and $(\sfrac{1}{2})_n$ (resp. $(0)_n$) is the $n$-dimensional vector whose entries are all $\sfrac{1}{2}$ (resp. $0$). 
By~\eqref{eq:reducedeq-for-Pi},
$\Pi_{N,\alpha} = \left\{x\in \RR^{A} \mid P x\leq c\right\}$. 
By Lemma~\ref{lem:FCrep}, the matrix $F$ is totally unimodular. It follows that the matrix $P$ is also totally unimodular. 
Any vertex $v$ of $\Pi_{N,\alpha}$ is the solution of a linear equation of the form $P'\,v=c'$, where $P'$ and $c'$ are obtained from $P$ and $c$ by deleting some rows, and $P'$ is invertible. Since $P'$ is unimodular, ${P'}^{-1}$ has integer coefficients, and $v={P'}^{-1}c'$ has half-integer coefficients.
\end{proof}

We are finally ready to prove the existence of the characteristic polynomials (Lemma~\ref{lem:existence-chrom}) and of the $A$-polynomial (Theorem~\ref{thm:Aqloring}).
Let $N=(A,\mfrk{C})$ be a regular oriented matroid. 
For all $\al\in \Om^\geq_N$, the polytope $\Pi_{N,\alpha}$ has vertices with half-integer coordinates. Hence by Ehrhart's theorem, there exists a polynomial $\chrom_{N,\al}^\geq$ such that for all odd positive integers~$q$, 
\begin{equation}\label{eq:defchialpha}
\chrom_{N,\al}^\geq(q)=|q\Pi_{N,\al}\cap \ZZ^A|.
\end{equation}
Moreover, denoting $\chrom_{N,\al}^>(X):=(-1)^{\dim(\Pi_{N,\al})}\chrom_{N,\al}^\geq(-X)$, Ehrhart's reciprocity gives
$$\chrom_{N,\al}^>(q)=|q\Pi_{N,\al}^\circ\cap \ZZ^A|.$$
By~\eqref{eq:weak-chrom-Ehrhart} the polynomial
\begin{equation*}
 \chrom_N^\geq:=\sum_{\al\in\Om_N^\geq}\chrom_{N,\al}^\geq
 \end{equation*}
is the weak-chromatic polynomial of $N$. Similarly, by~\eqref{eq:strict-chrom-Ehrhart} the polynomial
\begin{equation}\label{eq:strict-chrom-sum-alpha}
 \chrom_N^\geq:=\sum_{\al\in\Om_N^>}\chrom_{N,\al}^>
 \end{equation}
is the strict-chromatic polynomial of $N$. This completes the proof of Lemma~\ref{lem:existence-chrom}.

Lastly, in Theorem~\ref{thm:Expansions1}, we saw that $\Apoly_N(q,y,z)$ can be written in terms of $\chrom_{N}^\geq(q)$. 
To be more specific, we have shown in the proof of~\eqref{Exp1} that for all positive integer $q$, 
$$\sum_{f\in F_N(q)}y^{|f_A^>|}z^{|f_A^<|}=\sum_{R\uplus S\uplus T=A}(y-1)^{|S|}\,(z-1)^{|T|}\, q^{\rk(N)-\rk(N_{\setminus R})}\, \chrom^>_{{}_{-T}N_{\setminus R}}(q).$$
Since $\chrom^>_{{}_{-T}N_{\setminus R}}(q)$ is polynomial in $q$ for all $T,R\subseteq A$, this proves the existence of the $A$-polynomial stated in Theorem~\ref{thm:Aqloring}.

\medskip
\section{Reciprocity results for the $A$-polynomial}\label{Ehrhart2}
In this section we consider the evaluations of the $A$-polynomial obtained by specializing $q$ at negative integers. We show that these evaluations have interesting combinatorial interpretations. In particular, we show that the specialization $\Apoly_N(-1,y,z)$ contains several generating functions of interest (Theorem~\ref{thm:reciprocity-A}). We also derive a duality relation for the invariant $\Apoly_N(-1,y,z)$ (Theorem~\ref{thm:duality-A}). 

We start by proving a reciprocity result for the weak and strict-characteristic polynomials.
Recall that a \emph{flat} of a matroid $N=(A,\mscr{C})$ is a set $S\subseteq A$ such that $\rk(S\cup\{a\})>\rk(S)$ for all $a \in A\setminus S$. We call a subset $S\subseteq A$, of an oriented matroid $N=(A,\mfrk{C})$, a \emph{cyclic flat} if $S$ is a flat and the restriction of $N$ to $S$ (that is, the matroid $N_{\setminus (A\setminus S)}$) is totally cyclic. Equivalently, a flat $S$ is a cyclic flat if it is a union of positive circuits of $N$. 
\begin{thm}\label{thm:ChromDuality}
For any regular oriented matroid $N=(A,\mfrk{C})$,
\begin{equation*}
\chrom_{N}^\geq(-q) = \sum_{T\subseteq A,~ T\text{ cyclic flat}}(-1)^{\rk(N)-\rk(T)}\, \chrom_{N_{/T}}^>(q),
\end{equation*}
where the sum is over the cyclic flats of $N$. 
\end{thm}
In the special case where $N$ is acyclic, Theorem~\ref{thm:ChromDuality} gives 
$
\ds \chrom_N^>(-q) = (-1)^{\rk(N)} \chrom_{N}^\geq(q),
$
since in this case $\emptyset$ is the only cyclic flat of $N$.

\begin{proof}
We will use Ehrhart-Macdonald reciprocity, but we first need to relate the set $\Om_N^\geq$ to some sets of the form $\Om_{N'}^>$. 
Let 
$$\Pi_N'=\Pi_N\cap [0,1/2)^A =\big\{x\in\RR^A \mid \forall a\in A,~0\leq x_a<1/2,~\textrm{and}~\forall C\in \fC, ~\langle \ov C, x\rangle\in \ZZ\big\},$$
and let
$$\Om'_N=\{\alpha_x \mid x\in \Pi'_N\}\,\subseteq \Om_N^\geq.$$
It is clear that for all odd positive integer~$q$, one has $q\Pi_N'\cap \ZZ^A=q\Pi_N\cap \ZZ^A$, hence
$$\chrom^\geq_N(q)=|q\Pi_N'\cap \ZZ^A|=\sum_{\al\in \Om'_N}\chrom_{N,\al}^\geq(q),$$
where $\chrom_{N,\al}^\geq$ is the polynomial characterized by~\eqref{eq:defchialpha}.
This gives $\chrom^\geq_N=\sum_{\al\in \Om'_N}\chrom_{N,\al}^\geq$.

Next, we show that the set $\Om'_N$ is in bijection with 
$$\Om_N''=\{(T,\be)\mid T\subseteq A\textrm{ cyclic flat of }N,~\beta\in \Om_{N_{/T}}^>\}.$$
Recall that for $T\subseteq A$, the set of circuits of $N_{/T}$ is $\fC_{/T}=\{C\setminus T\mid C\in \fC,~\un C\not\subseteq T\}$, where $C\setminus T:=(C^+\setminus T,C^-\setminus T)$.
For $(T,\beta)\in\Om_N''$, let $\Phi(T,\be)$ be the map $\alpha:\fC\to \ZZ$ defined by $\alpha(C)=0$ if $\un C\subseteq T$, and $\alpha(C)=\beta(C\setminus T)$ otherwise. 
Observe that the map $\alpha:=\Phi(T,\be)$ is in $\Om'_N$. Indeed, there exists $y\in \Pi_{N_{/T}}^>$ such that $\alpha_y=\beta$, and then the point $x=(x_a)_{a\in A}$ defined by $x_a=y_a$ for $a$ in $A\setminus T$ and $x_a=0$ for $a$ in $T$ satisfies $x\in \Pi_N'$ and $\alpha_x=\alpha$.

We now prove that the map $\Phi:\Om_N''\to \Om_N'$ is a bijection. We start by proving that $\Phi$ is injective.
For $\alpha\in \Om'_N$, let $T_\alpha\subseteq A$ be the union of the positive circuits $C\in \fC$ such that $\alpha(C)=0$.
We claim that if $\alpha=\Phi(T,\be)$, then $T=T_\alpha$. 
To see that $T\subseteq T_\alpha$, recall that, by definition, $T$ is a cyclic flat, hence a union of positive circuits, and that every circuit $C$ contained in $T$ satisfies $\alpha(C)=0$. 
To show that $T_\alpha\subseteq T$, we consider a point $y$ in $\Pi_{N/T}^>$ such that $\alpha_y=\beta$. Then, for $a \notin T$ we observe that for any positive circuit $C\in \fC$ containing $a$, $\alpha(C)=\beta(C\setminus T)\geq y_a>0$; hence $a\notin T_\alpha$. 
Thus, if $\Phi(T',\be')=\alpha=\Phi(T,\be)$ then $T'=T_\alpha=T$, and in this case we obviously have $\be'=\be$. This proves that $\Phi$ is injective.

It remains to prove that $\Phi$ is surjective. Let $\alpha\in \Om'_N$.
Let $T:=T_\alpha\subseteq A$. Let $x\in \Pi'_{N}$ be such that $\alpha_x=\alpha$, and let $(y_a)_{a \in A\setminus T}$ be defined by $y_a=x_a$ for all $a\in A\setminus T$. Then $\be:=\al_y$ is a map on $\fC_{/T}$ defined by $\be(C)=\langle \ov C, y\rangle$. We will show that $(T,\be)$ is in $\Om_N''$ and $\Phi(T,\be)=\alpha$.

We first show that the map $\be$ has values in $\ZZ$ and only depends on $\al$ (not on $x$). 
First observe that $x_a=0$ for all $a\in T$ (since $a$ is in a positive circuit $C\in\fC$ such that $\langle C,x \rangle=0$). 
Recall that the circuits of $N_{/T}$ are $\{C\setminus T,~C\in \fC\}$. Moreover, for all $C\in \fC$,
$$\be(C\setminus T)=\langle \ov{C\setminus T},y\rangle =\langle C,x\rangle=\al(C)\in \ZZ,$$
which only depends on $\al$.

Next we construct another point $x'\in \Pi'_{N}$ such that $\alpha_{x'}=\alpha_x$ and $x_a\neq 0$ for all $a\notin T$. 
Let $P=\{a\in A,~x_a\neq 0 \}$. The set $P$ represents the coordinates that are positive in $x$, the set $T$ corresponds to some zero coordinates of $x$, and the set $Z:=A\setminus (T\cup P)$ corresponds to the extra zeroes of $x$ that we want to remove.
It follows easily from the definition of $T=T_{\al_x}$ that  
$$T=\{a\in A\mid a\textrm{ belongs to a positive circuit of $N_{\setminus P}$}\}.$$
Since any element in $A\setminus P$ is either in a positive circuit or a positive cocircuit of $N_{\setminus P}$, we conclude that for all $a\in Z$, there exists a positive cocircuit $D_a$ of $N_{\setminus P}$ such that $a\in D_a^+$. 
Hence, for all $a\in Z$, there exists a cocircuit $E_a$ of $N$ such that $a\in E_a^+$ and $E_a^-\subseteq P$ (since the cocircuits of $N_{\setminus P}$ are the signed sets $E\setminus P$ with $E$ a cocircuit of $N$).
Let $0<\eps<1/2$ be a positive real number smaller than $|x_a|$ and $|1/2-x_a|$ for all $a\in P$ and 
let $$\ds x'=x+\frac{\eps}{|Z|}\sum_{a\in Z }\ov E_a\in \RR^{A}.$$
By Lemma~\ref{lem:CircuitOrthogonal}, the incidence vectors of circuits and cocircuits of $N$ are orthogonal, so $\al_{x'}=\al_x$. As before we have $x_a'=0$ for all $a\in T$.
Moreover, by construction, $x_a'\in (0,1/2)$ for all $a\notin T$. In conclusion, $\alpha_{x'}=\alpha$, $\be(C)=\langle \ov C, x'\rangle$, and $T=\{a\in A\mid x_a'= 0\}$.

Let $y'=(y_a')_{a \in A\setminus T}$ be defined by $y_a'=x_a'$ for all $a\in A\setminus T$. We have $\be=\al_{y}=\al_{y'}$, and $y'\in (0,1/2)^{A\setminus T}$ hence $\be\in \Om_{N_{/T}}^>$. By definition, $T=T_\al$ is a union of positive circuits. 
Moreover $T$ is a flat of $N$. Indeed, if we suppose for contradiction that there exists $a\notin T$ such that $\rk(T\cup \{a\})=\rk(T)$, then there is $C\in \fC$ such that $a\in \un C\subseteq T\cup \{a\}$, and we get $\al(C)=x_a'\notin \ZZ$ which is a contradiction. Hence $T$ is a cyclic flat.
This completes the proof that $(T,\be)$ is in $\Om_N''$.
Lastly, it is clear that $\Phi(T,\be)=\alpha$. This proves that $\Phi$ is surjective, hence a bijection.

We can now complete the proof of Theorem~\ref{thm:ChromDuality}. Since $\Phi$ is a bijection between $\Om_N'$ and $\Om_N''$, we get 
\[
\chrom_{N}^\geq = \sum_{\alpha\in\Omega_N^\geq} \chrom_{N,\alpha}^\geq
=\sum_{(T,\be)\in \Om}\chrom_{N,\Phi(T,\be)}^\geq.
\]
Moreover, if $\alpha=\Phi(T,\be)$, then for all $x\in \Pi_{N, \al}$ one has $x_a=0$ for all $a\in T$, and furthermore the polytopes $\Pi_{N, \al}$ and $\Pi_{N_{/T}, \be}$ are in bijection by the canonical projection from $\RR^A$ to $\RR^{A\setminus T}$. Hence for all odd positive integers $q$,
$$\chrom_{N,\Phi(T,\be)}^\geq(q)=|q\Pi_{N, \al}\cap \ZZ^A|=|q\Pi_{N_{/T}, \be}\cap \ZZ^{A\setminus T}|=\chrom_{N_{/T},\be}^\geq(q).$$
This gives $\chrom_{N,\Phi(T,\be)}^\geq=\chrom_{N_{/T},\be}^\geq$, and 
$$\chrom_{N}^\geq=\sum_{T\subseteq A\text{ cyclic flat of }N}~\sum_{\be\in \Omega_{N_{/T}}^>}\chrom_{N_{/T},\be}^\geq.$$

By Lemma~\ref{lem:HalfInteger}, we know that $\dim(\Pi_{N_{/T},\be})=\rk(N_{/T})=\rk(N)-\rk(T)$ for all $\be\in \Omega_{N_{/T}}^>$. Hence Ehrhart-Macdonald reciprocity gives 
$$\ds \chrom_{N_{/T},\be}^\geq(-q)=(-1)^{\rk(N)-\rk(T)}\big|q\Pi_{N_{/T},\be}^\circ\cap \ZZ^{A\setminus T}\big|=(-1)^{\rk(N)-\rk(T)}\chrom_{N_{/T},\be}^>(q),$$
and finally
\[\chrom_{N}^\geq(-q) =\sum_{\substack{T\subseteq A\\\text{ cyclic flat }}}\sum_{\be\in \Omega_{N_{/T}}^>} (-1)^{\rk(N)-\rk(T)}\chrom_{N_{/T},\be}^>(q)\\
=\sum_{\substack{T\subseteq A\\\text{ cyclic flat }}}(-1)^{\rk(N)-\rk(T)}\chrom_{N_{/T}}^>(q).\qedhere
\]
\end{proof}

Next we use Theorem~\ref{thm:ChromDuality} to give simple interpretations for the evaluations $\chrom_N^\geq(-1)$ and~$\chrom_N^>(-1)$.

\begin{thm}\label{thm:Indicators}
For any regular oriented matroid $N$,
\[\chrom_N^\geq(-1) = \begin{cases}
1&\text{if $N$ is totally cyclic}\\
0&\text{otherwise},
\end{cases}\]
and
\[\chrom_{N}^>(-1) = \begin{cases}
(-1)^{\rk(N)} &\text{if $N$ is acyclic}\\
0&\text{otherwise}.
\end{cases}\]
\end{thm}
Before we give the proof of Theorem~\ref{thm:Indicators}, we need one more lemma.

\begin{lemma}\label{lem:AlphaZero}
Let $N=(A,\mfrk{C})$ be a regular oriented matroid. The matroid $N$ is acyclic if and only if $\Omega_N^>$ contains the \emph{zero map} (the map associating 0 to every circuit).
\end{lemma}
\begin{proof}
 Suppose first that $N$ is not acyclic. Let $C\in\mfrk{C}$ be a positive circuit. Then for all $x\in\Pi_{N}^>$, $\alpha_x(C)=\sum_{a\in C}x_a>0$, hence $\alpha_x\neq 0$. 

 Suppose now that $N$ is acyclic. We want to find $x\in \Pi_{N}^>$ such that $\alpha_x=0$. For every element $a\in A$ there is a positive cocircuit $D_a$ containing $a$. Let $x=(x_a)_{a\in A}\in \RR^A$ be given by $$x=\eps\, \sum_{a\in A}\ov{D_a},$$
 where $0<\eps<\frac{1}{2|A|}$, and $\ov{D_a}\in\{0,1\}^A$ is the incidence vector of $D_a$. By construction, $0<x_a<1/2$ for all $a\in A$. Moreover, for all $C\in \fC$, and all $a\in A$, one has $\langle \ov{C},\ov{D_a}\rangle=0$ since the incidence vectors of circuits and cocircuits of $N$ are orthogonal by Lemma~\ref{lem:CircuitOrthogonal}. Thus for all $C\in \fC$,
$$ 
\al_x(C)=\langle \ov C, x\rangle=\eps\,\sum_{a\in A}\langle \ov{C},\ov{D_a}\rangle=0.$$
Hence, $x$ is in $\Pi^>_N$, and the zero map $\alpha_x$ is in $\Omega_N^>$.
\end{proof}

\begin{proof}[Proof of Theorem~\ref{thm:Indicators}.]
Notice that $\chrom_N^>(1) = \ONE_{A=\emptyset}$. Combining this with Theorem~\ref{thm:ChromDuality} gives
\[\chrom_{N}^\geq(-1)
=\sum_{\substack{T\subseteq A\\\text{ cyclic flat}}}(-1)^{\rk(N)-\rk(T)}\chrom_{N_{/T}}^>(1)
=\sum_{\substack{T\subseteq A\\\text{ cyclic flat}}}(-1)^{\rk(N)-\rk(T)}\ONE_{T=A}
=\ONE_{\text{$N$ is totally cyclic}}.\]

By Ehrhart-McDonald reciprocity, and Lemma~\ref{lem:HalfInteger} we get $\chrom_{N,\alpha}^>(-1)=(-1)^{\rk(N)}\chrom_{N,\alpha}^\geq(1)$ for all $\al\in \Omega_N^>$. Hence
\[\chrom_N^>(-1) 
= \sum_{\alpha\in \Omega_N^>}\chrom_{N,\alpha}^>(-1)
=\sum_{\alpha\in \Omega_N^>}(-1)^{\rk(N)}\chrom_{N,\alpha}^\geq(1).\]
Next we observe that $\chrom_{N,\alpha}^\geq(1)=\ONE_{\alpha=0}$, where 0 denotes the zero map. By Lemma~\ref{lem:AlphaZero}, we know that $0\in \Omega_N^>$ if and only if $N$ is acyclic. Altogether, this gives $\ds \chrom_N^>(-1) = (-1)^{\rk(N)}\ONE_{\text{$N$ is acyclic}}$.
\end{proof}

We can now state our reciprocity result for the $A$-polynomial. 

\begin{thm}\label{thm:reciprocity-A}
For any regular oriented matroid $N=(A,\mfrk{C})$, 
\begin{align}
\sum_{\substack{R\uplus S\uplus T=A\\ \text{${}_{-T}N_{\setminus R}$ acyclic}}} 
y^{|S|}z^{|T|}
&=(-1)^{\rk(N)}\Apoly_N(-1,1+y,1+z),\label{Exp12}\\
\sum_{\substack{R\uplus S\uplus T=A\\ \text{${}_{-T}N_{\setminus R}$ totally cyclic}}}
y^{|S|}z^{|T|}(-1)^{\rk(N_{\setminus R})}
&=(-1)^{\rk(N)}(1+y+z)^{|A|}\Apoly_N\left(-1,\frac{1+y}{1+y+z},\frac{1+z}{1+y+z}\right),\nonumber\\
\sum_{\substack{R\uplus S\uplus T=A\\ \text{${}_{-T}N_{/ R}$ acyclic}}}
y^{|S|}z^{|T|}(-1)^{\rk(N_{/R})}
&=\Apoly_N(-1,y,z),\nonumber\\
\sum_{\substack{R\uplus S\uplus T=A\\ \text{${}_{-T}N_{/ R}$ totally cyclic}}}
 y^{|S|}z^{|T|}
&=(1+y+z)^{|A|}\Apoly_N\left(-1,\frac{y}{1+y+z},\frac{z}{1+y+z}\right),\label{Exp42}
\end{align}
\end{thm}

\begin{proof}
The above identities follow from combining~\eqref{Exp1}-\eqref{Exp4} and Theorem~\ref{thm:Indicators}.
\end{proof}

Let us comment on the significance of Theorem~\ref{thm:reciprocity-A}. The left-hand side of~\eqref{Exp12} is the generating function of the acyclic orientations that can be obtained from $N$ by deleting or reorienting some edges. Hence the $A$-polynomial of $N$ captures this information about $N$. For instance, the number of acyclic suborientations of $N$ is $(-1)^{\rk(N)}\Apoly_N(-1,2,1)$, while the number of acyclic reorientations of $N$ is $[y^{|A|}](-1)^{\rk(N)}\Apoly_N(-1,1+y,1+y)=[y^{|A|}](-1)^{\rk(N)}\Apoly_N(-1,y,y)$.
More generally the generating function of the acyclic suborientations of $N=(A,\fC)$ counted by number of ground set elements is
$$\sum_{R\subseteq A,~ \text{$N_{\setminus R}$ acyclic}}y^{|A\setminus R|}=(-1)^{\rk(N)}\Apoly_N(-1,1+y,1),$$
and the generating function of the acyclic reorientations of $N$ counted by number of reoriented elements is
$$\sum_{T\subseteq A,~ \text{${}_{-T}N$ acyclic}} z^{|T|} =(-1)^{\rk(N)}[y^{|A|}]\Apoly_N(-1,y,yz).$$
Similarly,~\eqref{Exp42} shows that the $A$-polynomial contains the generating functions of the totally-cyclic contractions of $N$ and the totally-cyclic reorientations of $N$.

\begin{example}
For the matroid $N=N_D$ corresponding to the digraph $D$ considered in Example~\ref{exp:Apoly}, one gets $$(-1)^{\rk(N)}\Apoly_N(-1,1+y,1)=y^3+5y^2+4y+1.$$ 
This corresponds to the fact that there is 1 acyclic subgraph of $N$ with 3 arcs, 5 acyclic subgraphs with $2$ arcs, 4 acyclic subgraphs with 1 arc, and 1 acyclic subgraph with no arcs. We also get 
 $$(-1)^{\rk(N)}[y^{|A|}]\Apoly_N(-1,y,yz)=z+4z^2+z^3,$$
which corresponds to the fact that there is 1 way of getting an acyclic graph by reorienting 1 arc of~$D$, 4 ways of getting an acyclic graph by reorienting 2 arcs, and 1 way of getting an acyclic graph by reorienting 3 arcs. 
\end{example}

Next we establish a duality relation for the invariant $\Apoly_N(-1,y,z)$.
\begin{thm}[Duality relation]\label{thm:duality-A}
Let $N$ be a regular oriented matroid, and let $N^*$ be the dual oriented matroid. The polynomials $\Apoly_N(-1,y,z)$ and $\Apoly_{N^*}(-1,y,z)$ are related by the following change of variables:
\begin{equation}
\Apoly_{N^*}(-1,y,z) = (-1)^{\rk(N^*)} (y+z-1)^{|A|}\Apoly_{N}\left(-1,\frac{y-1}{y+z-1},\frac{z-1}{y+z-1}\right).
\end{equation}
\end{thm}
\begin{proof}
Recall that an oriented matroid $N$ is acyclic if and only if $N^*$ is totally cyclic, and that for $S\subseteq A$, $(N_{\setminus S})^* = (N^*)_{/S}$ and $({}_{-S}N)^* = {}_{-S}(N^*)$. 
Hence, comparing~\eqref{Exp12} for $N$ with~\eqref{Exp42} for $N^*$ gives
\begin{align*}
(-1)^{\rk(N)} \Apoly_N(-1,1+y,1+z)&=\sum_{\substack{R\uplus S\uplus T=A\\ \text{${}_{-T}N_{\setminus R}$ acyclic}}} 
y^{|S|}z^{|T|}\\
&=\sum_{\substack{R\uplus S\uplus T=A\\ \text{${}_{-T}(N^*)_{/ R}$ totally cyclic}}}
 y^{|S|}z^{|T|}\\
 &=(1+y+z)^{|A|}\Apoly_{N^*}\left(-1,\frac{y}{1+y+z},\frac{z}{1+y+z}\right).\qedhere
\end{align*}
\end{proof}

In~\cite{Awan-Bernardi:B-poly} a similar duality relation was established for the specialization $\Bpoly_D(-1,y,z)$ of the $B$-polynomial of a digraph $D$. This is no coincidence. Indeed, as we now show, the polynomials $\Bpoly_D(-1,y,z)$ and $\Apoly_D(-1,y,z)$ are equal up to a prefactor for any digraph $D$.
\begin{prop}\label{prop:A=B-at-1}
For any digraph $D=(V,A)$, 
\begin{equation}\label{A=B}
\Apoly_{S}(-1,y,z) = (-1)^{\comp(D)}\Bpoly_D(-1,y,z).
\end{equation}
\end{prop}
\begin{proof}
Theorem 6.8 in~\cite{Awan-Bernardi:B-poly} gives 
\[(-1)^{|V|} \Bpoly_D(-1,1+y,1+z)=\sum_{\substack{R\uplus S\uplus T=A \\\text{${}_{-T}(N_D)_{\setminus R}$ acyclic}}} y^{|S|}z^{|T|}\].
Comparing this with~\eqref{Exp12}, and using $|V|=\comp(D)+\rk(N_D)$, gives
\[(-1)^{\comp(D)}(-1)^{\rk(N_D)} \Bpoly_D(-1,1+y,1+z)=(-1)^{\rk(N_D)}\Apoly_{D}(-1,1+y,1+z),\]
which is equivalent to~\eqref{A=B}.
\end{proof}


In~\cite{Awan-Bernardi:B-poly} it was also observed that for any acyclic digraph $D$, there is a ``symmetry'' relation for the coefficients of $\Bpoly_D(q,y,1)$. This relation also holds for the invariant $\Apoly_N(q,y,1)$ of an acyclic regular matroid and reads:
\begin{equation}\label{Symmetry}
\Apoly_N(-q,y,1) = (-1)^{\rk(N)} y^{|A|}\Apoly_N(q,\sfrac{1}{y},1).
\end{equation}
Indeed, applying Theorem~\ref{thm:ChromDuality} to~\eqref{Exp1} gives
\[\Apoly_N(-q,y+1,1) = \sum_{R\uplus S\subset A} y^{|S|} \chrom_{N_{\setminus R}}^\geq(q) q^{\rk(N)-\rk(N_{\setminus R})} (-1)^{\rk(N)}.\]
Comparing this with~\eqref{Exp2} gives 
\[\Apoly_N(-q,y+1,1) = (-1)^{\rk(N)} (1+y)^{|A|}\Apoly_N\left(q,\frac{1}{1+y},1\right),\]
which is equivalent to~\eqref{Symmetry}.

\medskip
\section{Tutte polynomials of partially oriented matroids}\label{TuttePolynomials}
In this section we recast and complement some of our results in the context of \emph{partially-oriented matroids}. Partially oriented matroids are abstractions of partially oriented graphs such as the one represented in Figure~\ref{fig:partially-oriented}(a). We will introduce two invariants of partially oriented matroids, which are defined in terms of the polynomials $\Apoly_N(q,y,1)$. These invariants are generalizations of the Tutte polynomial of unoriented matroids.

\fig{width=\linewidth}{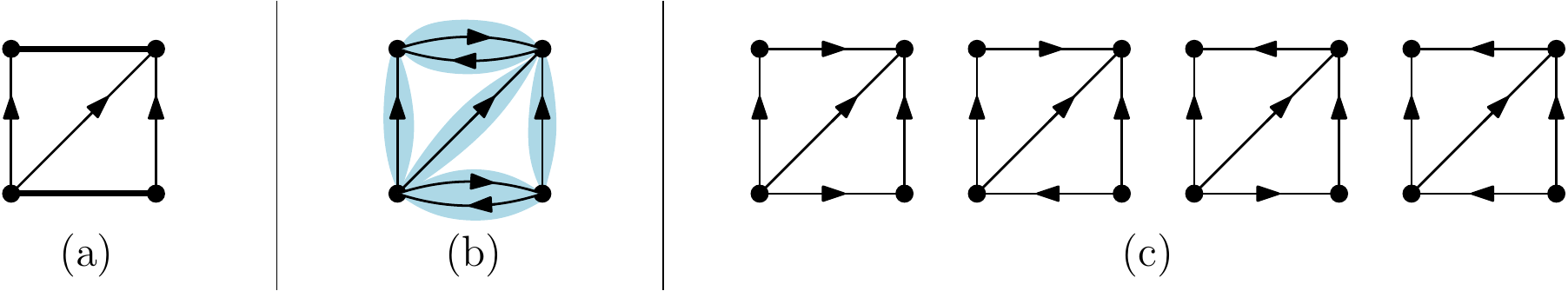}{(a) A partially oriented graph with 3 oriented elements and 2 unoriented elements. (b) The same partially oriented graph seen as a directed graph together with a partition of the arcs into singletons (oriented elements) and doubletons (unoriented elements). (c) The set of complete orientations of the partially oriented graph.}

A \emph{partially oriented matroid} $P=(N,E)$ is a pair made of an oriented matroid $N=(A,\mfrk{C})$ together with a \emph{ground-partition} $E$, which is a partition of $A$ into singletons, and doubletons made of two elements in $A$ which are opposite of each other in $N$. The singletons in $E$ are called \emph{oriented elements} of $P$ while doubletons are called \emph{unoriented elements}. This is represented in Figure~\ref{fig:partially-oriented} in the special case of a partially oriented graph.

Note that the ground-partition $E$ is not fully determined by $N$ (for instance, we do not require opposite elements in $A$ to be grouped into unoriented elements). The number of blocks in the set partition $E$ is denoted by $|E|$, so that 
$$|E| = \#\text{oriented elements}+\#\text{unoriented elements}.$$
A \emph{complete orientation} of a partially oriented matroid $P=(N,E)$ is an oriented matroid obtained from $N$ by deleting one element from each doubleton of $E$. This intuitively corresponds to choosing a direction for each of the unoriented elements of $P$. We denote by $\Orient(P)$ the set of complete orientations of $P$. 
It is easy to see that the unoriented matroids underlying the oriented matroids in $\Orient(P)$ are all isomorphic. We call \emph{unoriented matroid underlying $P$} the matroid $\un P$ with ground set $E$ which is isomorphic to $\un{\vec{N}}$ for all $\vec N\in \Orient(P)$. For instance, if $P$ corresponds to a partially oriented graph, then the unoriented matroid $\un P$ corresponds to the underlying unoriented graph. Note that $\rk(\un P)=\rk(N)$.

A partially oriented matroid $P=(N,E)$ is called \emph{regular} if $N$ is regular.
Clearly, the regular unoriented matroids can be identified with the regular partially oriented matroids which have no oriented elements. 
Indeed, to a regular unoriented matroid $M$, we have associated the oriented matroid $\orient{M}$. Since the elements of $\orient{M}$ come canonically in pairs (corresponding to the elements of $M$), there is a canonical ground-partition $E$ for $\orient{M}$ which is made only of doubletons. From now on, we view $\orient{M}$ as a partially oriented matroid with this canonical ground-partition $E$. Note that $\un{\orient{M}}=M$, and that $\Orient(\orient{M})$ is equal to the set $\Orient(M)$ of orientations of $M$ specified by Definition~\ref{def:OrientM}.

%
\begin{defn}
For a regular partially oriented matroid $P=(N,E)$ we define two invariants:
\begin{equation}\label{Tutte1}
T_P^{(1)}(x,y)\defeq \frac{y^{|E|}}{(y-1)^{\rk(N)}} \Apoly_N\left((x-1)(y-1),\frac{1}{y},1\right),
\end{equation}
and
\begin{equation}\label{Tutte2}
T_P^{(2)}(x,y)\defeq \frac{(y/2)^{|E|}}{(y-1)^{\rk(N)}}\sum_{\vec{N}\in \Orient(P)} \Apoly_{\vec{N}}\left((x-1)(y-1),\frac{2-y}{y},1\right).
\end{equation}
\end{defn}

\fig{width=\linewidth}{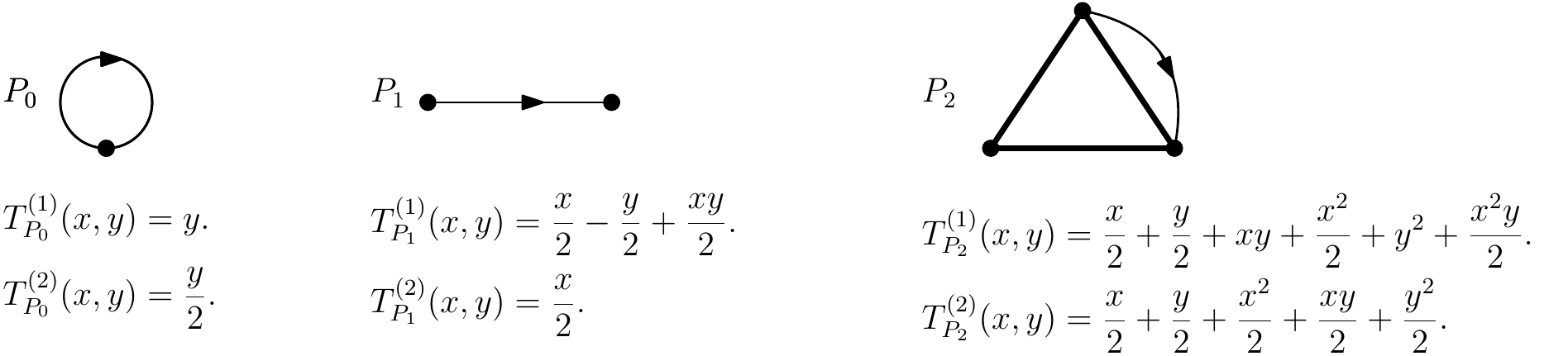}{Three partially oriented matroids (represented as partially oriented graphs) and their invariants $T_P^{(1)}$ and $T_P^{(2)}$.}

The invariants $T_P^{(1)}$ and $T_P^{(2)}$ are computed for three partially oriented matroids in Figure~\ref{fig:partial-Tutte-exp}. 
The relations~\eqref{eqOrient} and~\eqref{eqAverage} between the $A$-polynomial and the Tutte polynomial translate into the following identities for $T_P^{(1)}$ and $T_P^{(2)}$.
\begin{cor}
For any regular unoriented matroid $M$, 
$$T_{\orient{M}}^{(1)}(x,y) = T_{\orient{M}}^{(2)}(x,y) = T_M(x,y).$$
\end{cor}

The preceding result shows that the invariants $T_P^{(1)}$ and $T_P^{(2)}$ are generalizations of the Tutte polynomial to regular partially oriented matroids.
Let us also prove that $T_P^{(1)}$ and $T_P^{(2)}$ are \emph{polynomials} in $x$ and $y$.
\begin{prop}\label{prop:TutteArePolynomials}
For any regular partially oriented matroid $P=(N,E)$ with ground-partition $E(N)$, the invariants $T_P^{(1)}(x,y)$ and $T_P^{(2)}(x,y)$ are polynomials in $x$ and $y$.
\end{prop}
 
\begin{proof}
It is clear that $T_P^{(1)}(x,y)$ and $T_P^{(2)}(x,y)$ are rational functions in $x,y$. To prove polynomiality, we consider the expansions of these invariants
in terms of strict characteristic polynomials. Using~\eqref{Exp1} we get
\begin{align}
T_P^{(1)}(x,y)&=\sum_{S\subseteq A} (-1)^{|A\setminus S|}(x-1)^{\rk(N)-\rk(N_{\setminus S})} y^{|E|-|A\setminus S|} (y-1)^{|A\setminus S|-\rk(N_{\setminus S})}\chrom_{N_{\setminus S}}^>(q),\label{TutteExp1}\\
T_P^{(2)}(x,y)&=\sum_{\vec{N}=(\vec{A},\vec{\mfrk{C}})\in\Orient(P)}~\sum_{S\subseteq \vec{A}} (-1)^{|\vec A\setminus S|} (x-1)^{\rk(N)-\rk(\vec{N}_{\setminus S})} \big(\frac{y}{2}\big)^{|S|}
(y-1)^{|\vec A\setminus S|-\rk(\vec{N}_{\setminus S})}
\chrom_{\vec{N}_{\setminus S}}^>(q),\label{TutteExp3}
\end{align}
where $A$ is the ground set of $N$ and $q=(x-1)(y-1)$.

In the expression of $T_P^{(2)}(x,y)$ all the exponents are clearly non-negative, so $T_P^{(2)}(x,y)$ is a polynomial. In the expression of $T_P^{(1)}(x,y)$, the exponents of $(x-1)$ and $(y-1)$ are clearly non-negative and we focus on the exponent of $y$. Observe that if the oriented matroid $N_{\setminus S}$ contains two opposite elements, then $\chrom_{N_{\setminus S}}^>(q)=0$. Hence in the expression of $T_P^{(1)}(x,y)$ we can restrict the sum to the subsets $S\subseteq A$ that contain at least one element of each doubleton of $E$. The exponent of $y$ is always positive for those terms, hence $T^{(1)}_N(x,y)$ is a polynomial. 
\end{proof}



Equations~\eqref{TutteExp1} and~\eqref{TutteExp3} simplify greatly for $y=0$ and one obtains the following result.
\begin{lemma}
For any regular partially oriented matroid $P=(N,E)$,
\begin{equation}\label{TutteYZero}
T^{(1)}_P(x,0)=T^{(2)}_P(x,0)=(-1)^{\rk(N)}\sum_{\vec{N}\in\Orient(P)}\chrom^>_{\vec{N}}(1-x).
\end{equation}
\end{lemma}

Equation~\eqref{TutteYZero} is an extension of the classical relation between the Tutte polynomial and characteristic polynomial of unoriented matroids (since for an unoriented matroid the sum in the right-hand side is equal to the characteristic polynomial evaluated at $1-x$ by Lemma~\ref{lem:characteristic}). Next, we generalize the classical interpretations of the Tutte polynomial evaluations $T_M(2,0)$ and $T_M(0,2)$ as counting the acyclic and totally cyclic orientations of $M$~\cite{LasVergnas:Tutte(02)}. 

\begin{prop}\label{thm:Orientations}
For any regular partially oriented matroid $P$,
\begin{equation}\label{AcyclicOrientations}
T_P^{(1)}(2,0)=T_P^{(2)}(2,0) = \#\text{acyclic complete orientations of $P$},
\end{equation}
\begin{equation}\label{CyclicOrientations}
T_P^{(2)}(0,2) = \#\text{totally cyclic complete orientations of $P$}.
\end{equation}
\end{prop}
\begin{proof}
Equation~\eqref{TutteYZero} combined with Theorem~\ref{thm:Indicators} gives~\eqref{AcyclicOrientations}. To get~\eqref{CyclicOrientations} we express $T_P^{(2)}(x,y)$ in terms of the weak chromatic polynomial using~\eqref{Exp2}:
$$ T_P^{(2)}(x,y)=\!\!
\sum_{\vec{N}=(\vec{A},\vec{\mfrk{C}})\in\Orient(P)}\,\sum_{S\subseteq \vec{A}}(x-1)^{\rk(N)-\rk(\vec{N}_{\setminus S})} \bigg(1-\frac{y}{2}\bigg)^{|S|}(y-1)^{|\vec{A}\setminus S|-\rk(\vec{N}_{\setminus S})} \chrom_{\vec{N}_{\setminus S}}^\geq((x-1)(y-1)).$$
Setting $x=0$ and $y=2$, and using Theorem~\ref{thm:Indicators} gives~\eqref{CyclicOrientations}. 
\end{proof}

\ob{Next, we give expressions for $T_P^{(1)}(x,1)$ and $T_P^{(2)}(x,1)$.
Recall that for an unoriented matroid $M=(E,\mscr{C})$, the Tutte polynomial evaluation $T_M(x,1)$ is the generating function of the independent sets of $M$ counted according to their rank. More precisely, by specializing~\eqref{eq:defTutte}, one gets
$$T_M(x,1)=(x-1)^{\rk(N)}\sum_{\substack{F\subseteq E\\\text{$F$ independent in $M$}}} \left(\frac{1}{x-1}\right)^{|F|}.$$
We obtain the following generalizations for partially oriented matroids.}

\begin{prop}
For any regular partially oriented matroid $P=(N,E)$ with ground set $A$,
\begin{equation}\label{TutteAtOne1}
T_P^{(1)}(x,1) = (x-1)^{\rk(N)}\sum_{\substack{F\subseteq A\\\text{$F$ independent in $N$}}} \left(\frac{1}{2(x-1)}\right)^{|F|},
\end{equation}
\begin{equation}\label{TutteAtOne2}
T_P^{(2)}(x,1) = \frac{(x-1)^{\rk(N)}}{2^{\#\textrm{ oriented elements of $P$}}}\sum_{\substack{F\subseteq E\\\text{$F$ independent in $\un P$}}}\left(\frac{1}{x-1}\right)^{|F|}.
\end{equation}
\end{prop}
It is not hard to see that~\eqref{TutteAtOne1} and~\eqref{TutteAtOne2} can equivalently be stated as follows:
\begin{eqnarray*}
T_P^{(1)}(x+1,1)& =& \sum_{\substack{\text{$F\subseteq E$ independent in $\un P$}}} 2^{-\#\textrm{ oriented elements in $F$}}\quad x^{\rk(\un P)-\rk(F)},\\
T_P^{(2)}(x+1,1) &=& 2^{-\#\textrm{ oriented elements in $P$}}\sum_{\substack{\text{$F\subseteq E$ independent in $\un P$}}}x^{\rk(\un P)-\rk(F)}.
\end{eqnarray*}

\begin{proof}
Setting $y=1$ in~\eqref{TutteExp1} gives
$$T_P^{(1)}(x,1) = \sum_{S\subseteq A,~ A\setminus S\textrm{ independent in }N}(-1)^{|A\setminus S|}(x-1)^{\rk(N)-|A\setminus S|}\chrom_{N_{\setminus S}}^>(0).$$
Moreover, any oriented matroid $N'=(A',\fC')$ which is independent (in the sense that $A'$ is independent) satisfies $\chrom_{N'}^>(q)=\left(\frac{q-1}{2}\right)^{|A'|}$. Hence, the above expressions gives
$$T_P^{(1)}(x,1) = \sum_{F\subseteq A,~ F\textrm{ independent in }N}(-1)^{|F|}(x-1)^{\rk(N)-|F|}(-1/2)^{|F|},$$
which is equivalent to~\eqref{TutteAtOne1}.
Similarly, setting $y=1$ in~\eqref{TutteExp3} gives
$$T_P^{(2)}(x,1)=\sum_{\vec{N}'=(\vec{A}',\vec{\mfrk{C}}')\in\Orient(P)~}\quad\sum_{F\subseteq \vec{A}',~F\textrm{ independent in }\vec N} (-1)^{|F|} (x-1)^{\rk(N)-|F|} (1/2)^{|E|-|F|} (-1/2)^{|F|}.$$
Lastly, each inner sum is equal to $\ds 2^{-|E|}\sum_{F\subseteq E,~\text{$F$ independent in $\un P$}}(x-1)^{\rk(N)-|F|}$, and the number of inner sums is $2^{\#\textrm{ unoriented elements of }P}$. This gives~\eqref{TutteAtOne2}.
\end{proof}


We conclude this section with two results: a generalization of the subset expansion of the Tutte polynomial~\eqref{eq:defTutte} and a generalization of its basis expansion in terms of internal and external activities~\cite{Tutte:dichromate}.
These extensions are related to the recurrences given in Lemma~\ref{lem:Recurrence2} and Lemma~\ref{lem:Recurrence1}, which we first translate in terms of $T_P^{(1)}$ and $T_P^{(2)}$.
For a partially oriented matroid $P=(N,E)$, and an unoriented element $e\in E$, we say that $e$ is a \emph{coloop} of $P$ if $e$ is a coloop of the underlying matroid $\un P$ (equivalently, $e$ is the support of a cocircuit of $N$) and we say that $e$ is a \emph{loop} of $P$ if it is a loop of $\un P$ (equivalently, both elements in $e$ are loops of $N$).
\begin{lemma}\label{lem:TutteRecurrence}
Let $P=(N,E)$ be a regular partially oriented matroid, and let $e\in E$ be an unoriented element of $P$. For $i$ in $\{1,2\}$ the following holds.

If $e$ is neither a coloop nor a loop of $P$, then
\begin{equation}\label{normalelement}
T_P^{(i)}(x,y) = T^{(i)}_{P_{\setminus e}}(x,y)+T^{(i)}_{P_{/e}}(x,y),
\end{equation}
where $P_{\setminus e}:=(N_{\setminus e}, E\setminus \{e\})$ and $P_{/e}:=(N_{/e}, E\setminus \{e\})$.

If $e$ is a coloop of $P$, then
\begin{equation}\label{SimpleColoop}
T_P^{(i)}(x,y) = x\,T^{(i)}_{P_{/e}}(x,y).
\end{equation}

If $e$ is a loop of $P$, then 
\begin{equation}\label{SimpleLoop}
T_P^{(i)}(x,y) = y\, T^{(i)}_{P_{\setminus e}}(x,y).
\end{equation}
\end{lemma}

\begin{proof}
Equations~\eqref{normalelement} and~\eqref{SimpleColoop} for $T^{(1)}_P$ are a direct translation of Lemma~\ref{lem:Recurrence1} (upon recalling that $\rk(N)=\rk(N_{/e})+1$ if $e$ is not a loop, and $\rk(N)=\rk(N_{\setminus e})+1$ if $e$ is a coloop and $\rk(N)=\rk(N_{\setminus e})$ otherwise). 
Equations~\eqref{normalelement} and~\eqref{SimpleColoop} for $T^{(2)}_P$ are a direct consequence of Lemma~\ref{lem:Recurrence2}: in the expression of $T_P^{(2)}$ one needs to group together the contributions of the pairs of complete orientations of $P$ differing only on $e$, and then apply Lemma~\ref{lem:Recurrence2} to each pair. 
Equation~\eqref{SimpleLoop} is easy to check using the property that $\Apoly_N=\Apoly_{N\setminus \{a\}}$ for any loop $a$ of $N$.
\end{proof}

\fig{width=.35\linewidth}{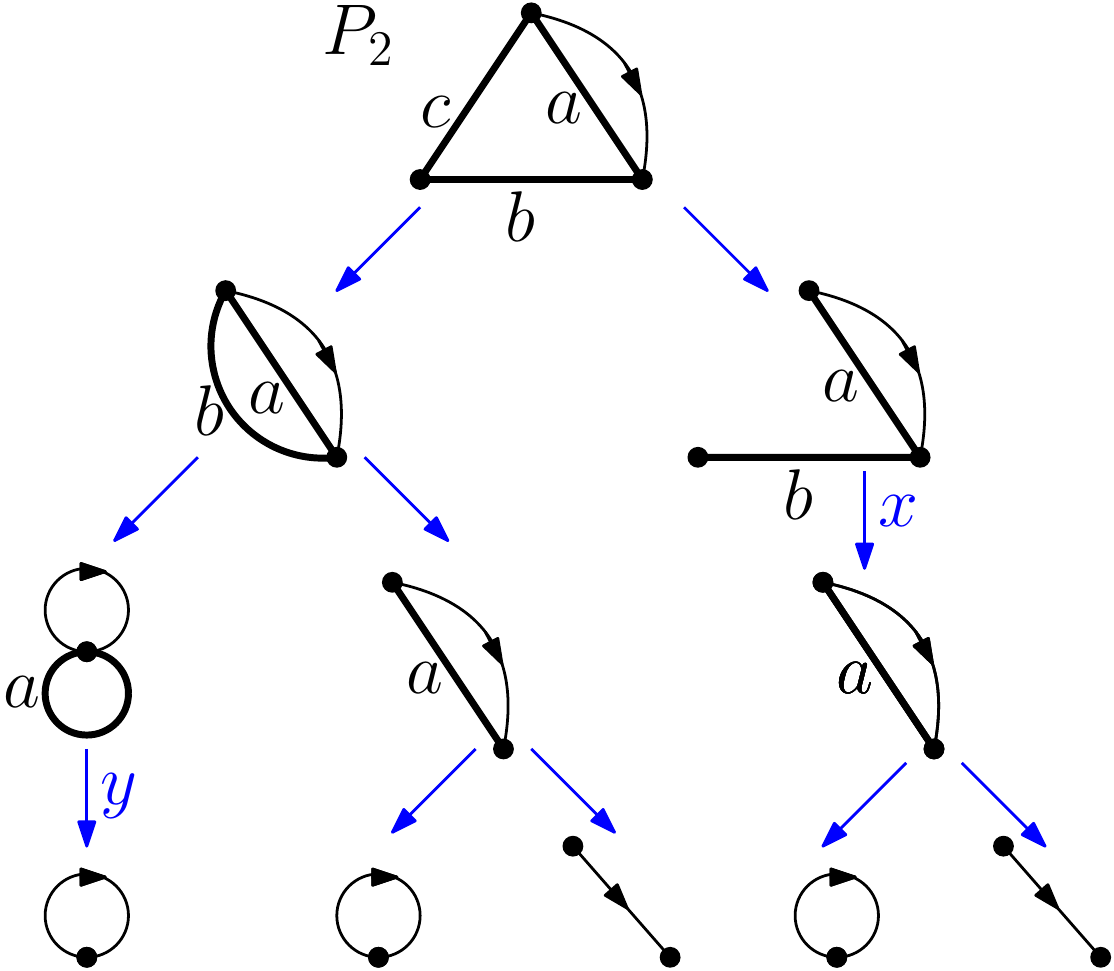}{Applying the recurrence relations of Lemma~\ref{lem:TutteRecurrence} to the unoriented elements $c$, then $b$, then $a$.}

\begin{example} 
By applying the recurrence relation (repeatedly) on the partially oriented matroid $P_2$ represented in Figure~\ref{fig:partial-Tutte-recurence}, one gets
$$T_{P_2}^{(i)}(x,y)=(x+y+1)\,T_{P_0}^{(i)}(x,y)+(x+1)\,T_{P_1}^{(i)}(x,y),$$
where $P_0$, $P_1$ are the partially oriented matroids represented in Figure~\ref{fig:partial-Tutte-exp}. 
\end{example}

We are ready to give a generalization of the subset expansion~\eqref{eq:defTutte} for partially oriented matroids.

\begin{prop}
Let $P=(N,E)$ be a regular partially oriented matroid, and let $H\subseteq E$ be its set of unoriented elements. For $i\in \{1,2\}$,
\begin{equation}\label{TutteSubgraph}
T_P^{(i)}(x,y) = \sum_{S\subseteq H} (x-1)^{\rk(P)-\rk(P_{\setminus \ov S})}(y-1)^{|S|-\rk(S)}T^{(i)}_{P_{\setminus \ov S/S}}(x,y),
\end{equation}
where  $\ov S=H\setminus S$, and $\rk(S)$ is the rank of $S$ in the underlying matroid $\un P$.
\end{prop}

\begin{proof}
We first rewrite the recurrence relations~\eqref{SimpleColoop} and~\eqref{SimpleLoop}. Observe that if $e\in H$ is either a loop or coloop of $P$, then $T_{P_{\setminus e}}^{(i)}=T_{P_{/ e}}^{(i)}$ for $i\in \{1,2\}$. Hence for a coloop $e$ of $P$ we get
\begin{equation}\label{ComplexColoop}
T_P^{(i)}(x,y) = (x-1)\,T^{(i)}_{P_{\setminus e}}(x,y)+T^{(i)}_{P_{/e}}(x,y),
\end{equation}
while for a loop $e$ of $P$ we get 
\begin{equation}\label{ComplexLoop}
T_P^{(i)}(x,y) = T^{(i)}_{P_{\setminus e}}(x,y)+(y-1)\,T^{(i)}_{P_{/e}}(x,y).
\end{equation}
We apply the recurrence relations~\eqref{normalelement},~\eqref{ComplexColoop} and~\eqref{ComplexLoop} successively on every unoriented element of $P$ (in an arbitrary order). 
This gives 
\[T^{(i)}_P(x,y)=\sum_{S\subseteq  H}(x-1)^{\al(S)}(y-1)^{\be(S)}T_{P_{\setminus \ov S /S}}^{(i)}(x,y),\]
where $S$ represents the elements contracted, $\ov S=H\setminus S$ represents the elements contracted, and $\al(S)$ is the number of coloops deleted during the deletion-contraction process leading from $P$ to $P_{\setminus \ov S /S}$, and $\be(S)$ is the number of loops contracted during this process. We observe that $\ds \al(S)=\rk(P)-\rk(P_{\setminus \ov S})$ since this number starts at 0 and increases (by one) exactly when deleting a coloop during the  deletion-contraction process, and $\be(S)=|S|-\rk(S)$ because this quantity starts at 0 and increases (by one) exactly when contracting a loop.
\end{proof}

Note that in the special case where $P$ has no oriented elements,~\eqref{TutteSubgraph} coincides with the subset expansion~\eqref{eq:defTutte} of the Tutte polynomial. The next result extends the basis expansion of the Tutte polynomials in terms of internal and external activities~\cite{Tutte:dichromate}. 

Let $P=(N,E)$ be a partially oriented matroid and let $H\subseteq E$ be its subset of unoriented edges. 
We say that a set $S\subseteq H$ is \emph{potentially spanning in $P$} if no cocircuit of $\un P$ has its support contained in $H\setminus S$. We say that  $S\subseteq H$ is a  \emph{potential basis} in $P$ if $S$ is independent in $\un P$ and potentially spanning in $P$. Equivalently, $S\subseteq H$ is a potential basis in $P$ if there is a set $S'\subseteq E\setminus H$ such that $S\cup S'$ is a basis of $\un P$.

\begin{prop}\label{prop:activities-partial}
Let $P=(N,E)$ be a regular partially oriented matroid, and let $H\subseteq E$ be its set of unoriented elements. Let $\prec$ be a linear ordering of $H$. For $i\in \{1,2\}$, 
\begin{equation}\label{TutteActivities}
T^{(i)}_P(x,y) = \sum_{\substack{B\subseteq H \\ B \textrm{ potential basis in $P$}}}
x^{|\interior_{\prec} (B)|}y^{|\exterior_{\prec}(B)|}\,T^{(i)}_{P_{\setminus \overline{B}/B}}(x,y),
\end{equation}
where $\overline{B}=H\setminus B$ and 
\begin{align*}
\interior_{\prec}(B) &= \{e\in B\mid \textrm{$e$ is the minimal element in a cocircuit of $\un P$ contained in $\overline{B}\cup \{e\}$}\},\\
\exterior_{\prec}(B)&=\{e\in \overline B\mid \textrm{$e$ is the minimal element in a circuit of $\un P$ contained in $B\cup \{e\}$}\}. 
\end{align*}
\end{prop}

Note that if $P$ has no oriented element, then the potential bases in $P$ are the bases of $\un P$. Moreover, in this case, $\interior_{\prec}(B)$ (resp. $\exterior_{\prec}(B)$) is the set of the elements which are internally (resp. externally) active with respect to the basis $B$ in the sense of Tutte~\cite{Tutte:dichromate}.
\begin{example}
In the case of the partially oriented matroid $P_2$ represented in Figure~\ref{fig:partial-Tutte-recurence}, there are 5 subsets of $H=\{a,b,c\}$ which are potential bases in $P$, namely $\{a,b\}$, $\{a,c\}$, $\{b,c\}$, $\{b\}$ and $\{c\}$. For the order $a\prec b \prec c$ the activities of the potential bases $B=\{a,b\}$ 
and $B'=\{b,c\}$ are the following: $\interior_{\prec}(B)=\{b\}$, $\exterior_{\prec}(B)=\emptyset$ and $\interior_{\prec}(B')=\emptyset$, $\exterior_{\prec}(B')=\{a\}$. 
Equation~\eqref{TutteActivities} gives 
$$T_{P_2}^{(i)}(x,y)=(x+y+1)\,T_{P_0}^{(i)}(x,y)+(x+1)\,T_{P_1}^{(i)}(x,y),$$
where $P_0$, $P_1$ are the partially oriented matroids represented in Figure~\ref{fig:partial-Tutte-exp}. 
\end{example}

\begin{proof}[Proof of Proposition~\ref{prop:activities-partial}]
We apply the recurrence relation of Lemma~\ref{lem:TutteRecurrence} successively on every unoriented element of~$H$ in the decreasing given by $\prec$ (that is, starting with the largest element in $H$ for $\prec$). This is represented in Figure~\ref{fig:partial-Tutte-recurence}.
This gives 
\[T^{(i)}_P(x,y)=\sum_{B\subseteq H\textrm{special}}x^{\de(B)}y^{\ga(B)}\,T_{P_{\setminus \ov B /B}}^{(i)}(x,y),\]
where the sum is over the subsets $B$ of $H$ such that no loop is contracted during the deletion-contraction process leading from $P$ to $P_{\setminus \ov B /B}$ and 
no coloop is deleted during this process; and $\de(B)$ is the number of coloops contracted and $\ga(B)$ is the number of loops deleted during this deletion-contraction process.
The fact that no loop is contracted during the process is equivalent to the fact that $B$ is independent, and the fact that no coloop is deleted is equivalent to the fact that $B$ is potentially spanning. Hence the sum is over the sets $B\subseteq H$ which are potential bases.
For such a set $B$, it is easy to see that $\de(B)=|\interior_{\prec}(F)|$ and $\ga(B)=|\exterior_{\prec}(F)|$.
\end{proof}

\medskip
\section{Cocycle reversing equivalence classes}\label{Q=0}
In Section~\ref{Ehrhart1}, we defined the weak and strict-characteristic polynomials of an oriented matroid $N$ in terms of the Ehrhart quasipolynomials of certain polytopes $\Pi_{N,\al}$. Here we interpret the vertices of these polytopes as reorientations of $N$, and show that the polytopes are in bijective correspondence with certain equivalence classes of orientations up to cocircuit reversing. The number of equivalence classes is also shown to be enumerated by an evaluation of the the characteristic \emph{quasi-polynomials}.

Let $N=(A,\mfrk{C})$ be a regular oriented matroid. A subset $S\subseteq A$ is a \emph{cocycle} of $\un N$ if $S$ is a disjoint union of cocircuits of $\un N$. The subset $S$ is a \emph{positive cocycle} of $N$ if $S$ is a disjoint union of (the support of) positive cocircuits of $N$. By convention the empty set is a positive cocycle. There are neat characterizations of cocycles and positive cocycles, as we now explain.

\begin{lemma}\label{lem:characterize-cocycles}
Let $N=(A,\mfrk{C})$ be a regular oriented matroid, and let $S\subseteq A$.
\begin{compactitem}
\item[(i)] $S$ is a cocycle of $\un N$ if and only if for all circuits $C\in \fC$, $|\un C\cap S|$ is even.
\item[(ii)] $S$ is a positive cocycle of $N$ if and only if for all circuits $C\in \fC$, $\langle \ov C,\ov S\rangle=0$, where $\ov C\in\{-1,0,1\}^A$ is the incidence vector of $C$, and $\ov S\in \{0,1\}^A$ is the indicator tuple of~$S$.
\end{compactitem}
\end{lemma}

\begin{proof}
Property (i) is well known for binary matroids, hence for regular matroids~\cite[~Theorem 9.1.1]{Oxley:matroid-theory}.
We now prove (ii). By Lemma~\ref{lem:CircuitOrthogonal} if $S$ is a positive cocycle, then $\langle \ov C,\ov S\rangle=0$ for all $C\in \fC$. Now suppose that $S\neq \emptyset$ satisfies $\langle \ov C,\ov S\rangle=0$ for all $C\in \fC$. We want to prove that $S$ is a disjoint union of positive cocircuits of $N$. 
Note that no circuit $C\in \fC$ satisfies both $C^+\cap S\neq \emptyset$ and $C^-\cap S= \emptyset$ (because this would contradict $\langle \ov C,\ov S\rangle=0$). Hence, the contraction $N_{/(A\setminus S)}$ has no positive circuit. This means that $N_{/(A\setminus S)}$ has a positive cocircuit $K$. By definition of contractions, $K$ is a positive cocircuit of $N$ contained in $S$. If $S=K$, then we are done. Otherwise we observe (by using Lemma~\ref{lem:CircuitOrthogonal} for $K$) that $S'=S\setminus K$ satisfies $\langle \ov C,\ov S'\rangle=0$ for all $C\in \fC$. Hence repeating this process allows one to write $S$ as a disjoint union of positive cocircuits of $N$. 
\end{proof}

\begin{cor} \label{cor:equiv-relation}
Let $N$ be a regular oriented matroid and let $N'$ be a reorientation.
\begin{compactitem}
\item[(i)]
There is a sequence $N=N_0,N_1,\ldots,N_k=N'$ where for all $i\in[k]$, $N_{i}$ is obtained from $N_{i-1}$ by reorienting a cocircuit of $\un N$ if and only if $N'$ is obtained from $N$ by reorienting a cocycle of $\un N$.
\item[(ii)] There is a sequence $N=N_0,N_1,\ldots,N_k=N'$ where for all $i\in[k]$, $N_{i}$ is obtained from $N_{i-1}$ by reorienting a positive cocircuit of $N_{i-1}$ if and only if $N'$ is obtained from $N$ by reorienting a positive cocycle of $N$.
\end{compactitem}
\end{cor}

Recall that the for an oriented matroid $N=(A,\fC)$ the circuits of ${}_{-S}N$ are ${{}_{-S}C}$ for $C\in \fC$, where
$${{}_{-S}C}:=\left(\,(C^+\setminus S)\cup(C^-\cap S),\,(C^-\setminus S)\cup(C^+\cap S)\,\right).$$ 


\begin{proof}
(i) Clearly if $N'$ is obtained from $N$ by reorienting a cocycle of $\un N$, then the sequence $N=N_0,N_1,\ldots,N_k=N'$ exists. To prove the reverse implication, it suffices to check that if $S,S'$ are cocycles of $\un N$, then so is the symmetric difference $S\triangle S'$. But this is easy to check from the characterization of cocycles given in Lemma~\ref{lem:characterize-cocycles}(i).
The proof of (ii) is similar: the forward implication is obvious and the reverse implication one needs to check that if $S$ is a positive cocycle of $N$ and $S'$ is a positive cocycle of ${}_{-S}N$, then $S\triangle S'$ is a positive cocycle of $N$. This follows easily from the characterization of positive cocycles given in Lemma~\ref{lem:characterize-cocycles}(ii) because for all $C\in \fC$,
$$\langle \ov{S\triangle S'},\ov C\rangle=\langle \ov S,\ov C\rangle+\langle \ov S',\ov{{}_{-S}C}\rangle.$$
\end{proof}

Let $N$ be a regular oriented matroid and let $N'$ be a reorientation. We write $N\sim N'$ if $N'$ is obtained from $N$ by reorienting a cocycle of $\un N$.
 We write $N\psim N'$ if $N'$ is obtained from $N$ by reorienting a positive cocycle of $N$. By Corollary~\ref{cor:equiv-relation} it is clear that these are equivalence relations. 
We denote by 
\begin{equation*}
\llbracket N\rrbracket:=\{N_1\mid N_1\sim N\},~\textrm{ and }~ [N]:=\{N_1\mid N_1\psim N\},
\end{equation*}
the corresponding equivalence classes.

\begin{example}
Consider the oriented matroid $N=N_D$ corresponding to the digraph $D$ represented in Figure~\ref{fig:cocycle-classes}(left). On the right of this figure, the equivalence class $\llbracket N\rrbracket$ is indicated and the three subclasses $[N']$ for $N'\in \llbracket N\rrbracket$ are circled by dotted lines.
\end{example}

The equivalences classes $[N]$ have been studied by Gioan for regular matroids~\cite{gioan2008circuit}. The set $[N]$ and is called the \emph{equivalence class of $N$ in the cocycle reversing system}. 

We note that if a regular oriented matroid $N$ is acyclic, then any reorientation in $[N]$ is also acyclic. Indeed,  if $S$ is a positive cocircuit of $N$, then none of the elements in $S$ belongs to a positive circuit of $N$ or ${}_{-S}N$ (since positive circuits and positive cocircuits are always disjoint), and consequently the positive circuits of $N$ are the same as the positive circuits of ${}_{-S}N$. Therefore, the ground set elements contained in a positive circuits are the same for $N$ and ${}_{-S}N$ for any positive cocircuit $S$, and in particular $N$ is acyclic if and only if ${}_{-S}N$ is acyclic. If the orientations in $[N]$ are acyclic, then we call $[N]$ an \emph{acyclic equivalence class in the cocycle reversing system}.

\fig{width=\linewidth}{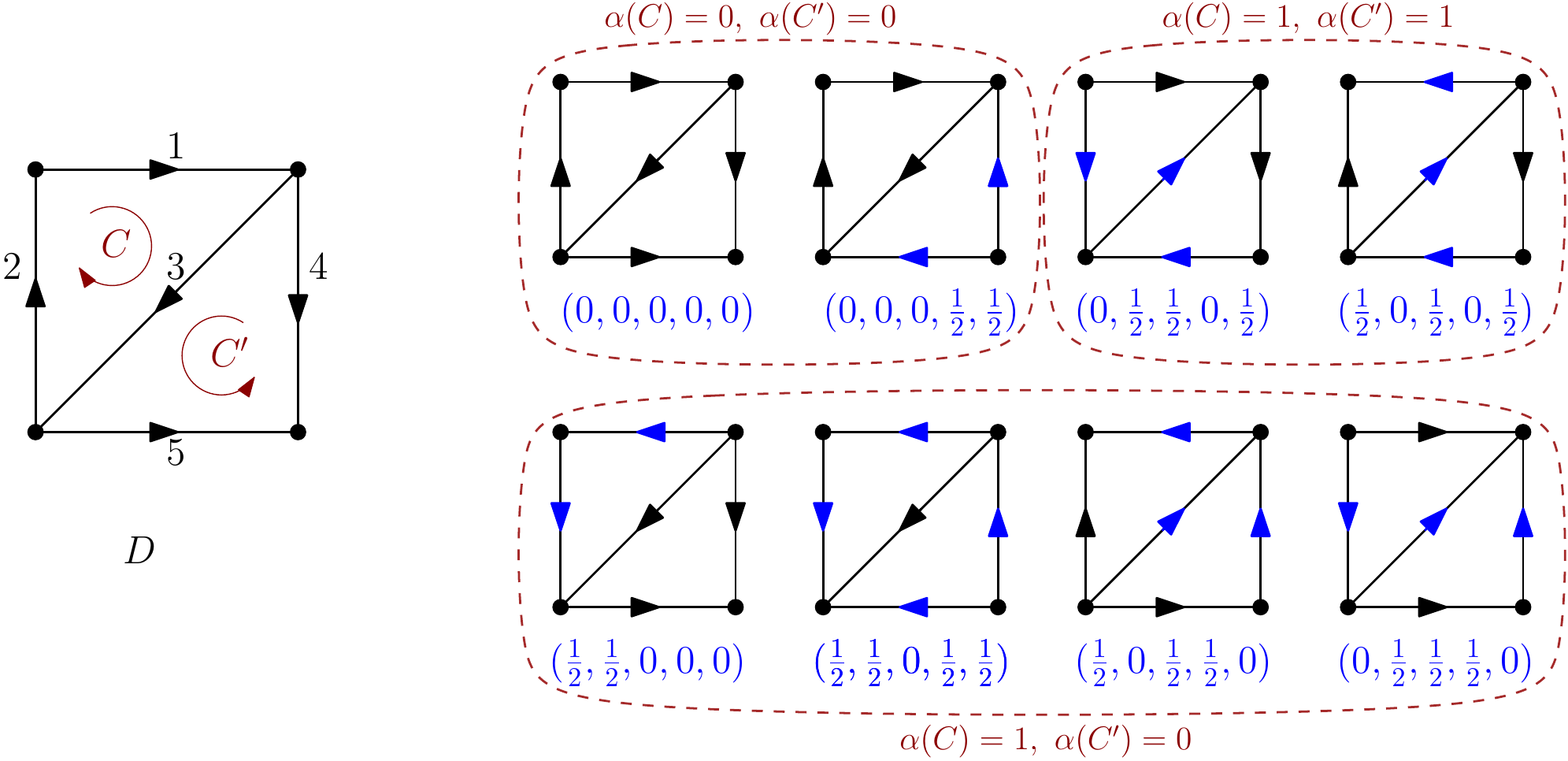}{Illustration of Theorems~\ref{thm:counting-cocycle-classes} and~\ref{thm:bijection-equiv-classes}.
Left: A digraph $D$ whose arc-set is $A=[5]$. The signed sets $C=(\{1,2,3\},\emptyset)$ and $C'=(\{3,5\},\{4\})$ are circuits of the oriented matroid $N=N_D$ associated to $D$. 
Right: The 8 digraphs obtained from $D$ by reorienting the arcs in a cocycle $S$ (of the underlying graph), and the point $x\in V_N$ such that $S(x)=S$. 
These digraphs form 3 equivalence classes of in the cocycle reversing system. The equivalence classes are in bijection with the polytopes in $\Pi_{N}$, and the digraphs correspond to the vertices of these polytopes. Precisely, each equivalence class is associated with a polytope  $\Pi_{N,\al}$ for some $\al\in \Om_N^\geq$ (and the value of $\alpha$ on $C$ and $C'$ is indicated in the figure), and the points $x\in V_N$ associated to the digraphs in this equivalence class are the vertices of the polytope $\Pi_{N,\al}$.
In this example, there is a unique equivalence class made of acyclic digraphs, which corresponds to the unique element $\al$ in $\Om_N^>$.}

We now state the main counting result of this section.

\begin{thm}\label{thm:counting-cocycle-classes}
For any regular oriented matroid $N=(A,\mfrk{C})$,
$$|\Om_N^\geq|= \#\left\{\,[{}_{-S}N] \mid S\subseteq A\textrm{ cocycle of }\un N\right\},$$
and
$$|\Om_N^>|=\#\left\{\,[{}_{-S}N] \mid S\subseteq A \textrm{ cocycle of }\un N \textrm{ such that }{}_{-S}N\textrm{ is acyclic} \right\}.$$
\end{thm}
Theorem~\ref{thm:counting-cocycle-classes} is reminiscent of a result of Gioan for unoriented regular matroids. Namely, it is shown in~\cite{gioan2008circuit} that for any unoriented regular matroid $M$ the number of equivalence classes of orientations of $N$ in the cocycle reversing system is given by the evaluation $T_M(1,2)$ of the Tutte polynomial, while the number of acyclic equivalence classes is given by the evaluation $T_M(1,0)$.
Equivalently, for a regular oriented matroid $N$,
$$T_{\un N}(1,2)= \#\{\,[{}_{-S}N] \mid S\subseteq A\}~\textrm{ and }~ T_{\un N}(1,0)= \#\{\,[{}_{-S}N] \mid S\subseteq A,\textrm{ such that }{}_{-S}N\textrm{ is acyclic} \}.$$
Recall that $T_{\un N}(1,2)$ and $T_{\un N}(1,0)$ are evaluations of the $A$-polynomial. 

\begin{remark}\label{rk:even-evaluation}
We do not know whether $|\Om_N^\geq|$ and $|\Om_N^>|$ are evaluations of the $A$-polynomial $\Apoly_N$. However we can relate these quantities to the Ehrhart polynomials associated to $\Pi_N$ as we now explain.
For $\al\in \Om_N^\geq$, let $E_N^\al$ be the quasipolynomial of the polytope $\Pi_{N,\al}$. 
As seen in Section~\ref{Ehrhart1}, these are quasipolynomials of period $2$, so there exist polynomials $E_{N,1}^\al$ and $E_{N,2}^\al$ such that for all non-negative integers $q$
$$E_{N,1}^\al(q)=|q\Pi_{N,\al}\cap \ZZ^A|\textrm{ if $q$ is odd}~\textrm{ and }~E_{N,2}^\al(q)=|q\Pi_{N,\al}\cap \ZZ^A|\textrm{ if $q$ is even}.$$
The polynomial $E_{N,1}^\al(q)$ is denoted $\chi_{N,\al}^\geq$ in Section~\ref{Ehrhart1} and the weak and strict characteristic polynomials are defined by
$$\chi_N^\geq(X)=\sum_{\al\in \Om_N^\geq}E_{N,1}^\al(X)~\textrm{ and }~\chi_N^>(X)=\sum_{\al\in \Om_N^>}(-1)^{\rk(N)}E_{N,1}^\al(-X).$$
Here we could consider the ``even'' versions of these polynomials:
$$\wt \chi_N^\geq(X)=\sum_{\al\in \Om_N^\geq}E_{N,2}^\al(X)~\textrm{ and }~\wt \chi_N^>(X)=\sum_{\al\in \Om_N^>}(-1)^{\rk(N)}E_{N,2}^\al(-X).$$
We then observe that $\wt \chi_N^\geq(0)=\sum_{\al\in \Om_N^\geq}|0\Pi_{N,\al}\cap \ZZ^A|=|\Om_N^\geq|$, and similarly $\wt \chi_N^>(0)=(-1)^{\rk(N)}|\Om_N^>|$, 
so that Theorem~\ref{thm:counting-cocycle-classes} can be restated as
\begin{eqnarray*}
\wt \chi_N^\geq(0)
&=& \#\{[{}_{-S}N] \mid S\subseteq A\textrm{ cocycle of }\un N\},\\
 (-1)^{\rk(N)}\wt\chi_N^>(0)
&=& \#\{[{}_{-S}N] \mid S\subseteq A\textrm{ cocycle of }\un N \textrm{ such that }{}_{-S}N\textrm{ is acyclic}\}.
\end{eqnarray*}
\end{remark}

In the rest of this section, we prove Theorem~\ref{thm:counting-cocycle-classes}.
Recall that $\Pi_N$ is a disjoint union of polytopes $\Pi_{N,\al}$ indexed by $\al\in \Om_N^\geq$. Recall from Lemma~\ref{lem:HalfInteger} that the vertices of every polytope $\Pi_{N,\al}$ have half-integer coefficients. Since $\Pi_N\subseteq [0,1/2]^A$, the vertices of the polytopes $\Pi_{N,\al}$ are exactly the points of $\Pi_{N,\al}$ with coordinates in $\{0,1/2\}$. Consider the set of all vertices of $\Pi_N$: 
$$V_N=\Pi_N\cap \{0,1/2\}^A=\biguplus_{\al \in\Om_N^\geq}\{x\textrm{ vertex of the polytope }\Pi_{N,\al}\}.$$
For a point $x\in \{0,1/2\}^A$, we define $\ds S(x):=\{a\in A\mid x_a=1/2\}\subseteq A$.

\begin{lemma}\label{lem:vertices-to-cocycles}
\begin{compactitem}
\item[(i)] A point $x\in \{0,1/2\}^A$ is in $V_N$ if and only if $S(x)$ is a cocycle of $\un N$.
\item[(ii)] The vertices $x,y\in V_N$ satisfy $\al_x=\al_y$ if and only if ${}_{-S(x)}N\psim {}_{-S(y)}N$.
\end{compactitem}
\end{lemma}

\begin{proof}
Let us first prove (i).
Let $x$ be in $\{0,1/2\}^A$, and let $S=S(x)$. By definition, $\langle x,\ov C\rangle=\frac12( |C^+\cap S|- |C^-\cap S|)$ for any signed subset $C$ of $A$. Hence $2\langle x,\ov C\rangle$ is an integer with the same parity as $|S\cap \un C|$. Hence $x$ is in $V_N$ if and only if $|S\cap \un C|$ is even for all $C\in \fC$, which is true if and only if $S$ is a cocycle of $N$ by Lemma~\ref{lem:characterize-cocycles}(i).

We now prove (ii). Let $x,y\in V_N$. We want to prove that $\al_x=\al_y$ if and only if the symmetric difference $S(x)\triangle S(y)$ is a positive cocycle of ${}_{-S(x)}N$.
Let $z$ be the point such that $S(z)=S(x)\triangle S(y)$. One can check that for any signed set $C$ from $A$,
$$\langle y,\ov C\rangle - \langle x,\ov C\rangle= \langle z,{}_{-S(x)}C\rangle.$$
Hence $\al_x=\al_y$ if and only if $\langle z,{}_{-S(x)}C\rangle =0$ for all $C$ in $\fC$. Since the circuit of ${}_{-S(x)}N$ are the signed sets of the form ${}_{-S(x)}C$ for $C\in \fC$, we conclude from Lemma~\ref{lem:characterize-cocycles}(ii) that $\al_x=\al_y$ if and only if $S(x)\triangle S(y)$ is a positive cocycle of ${}_{-S(x)}N$. 
\end{proof}

Lemma~\ref{lem:vertices-to-cocycles} is illustrated in Figure~\ref{fig:cocycle-classes}. Note that  Lemma~\ref{lem:vertices-to-cocycles}(i) gives a bijection between the vertices of the union of polytope $\Pi_N$ and the cocycles of $\un N$. 

By Lemma~\ref{lem:vertices-to-cocycles}, we can define a map 
$$\Psi:\Om_N^\geq\to \{[N'],~N'\in\llbracket N\rrbracket\}$$ 
as follows: for $\al\in \Om_N^\geq$ we consider any vertex $x$ of $\Pi_{N,\al}$ and define $\Psi(\al)=[{}_{-S(x)}N]$. Indeed, Lemma~\ref{lem:vertices-to-cocycles}(ii) shows that $\Psi(\al)$ is well defined (independent of the choice of the vertex $x$ of $\Pi_{N,\al}$) and Lemma~\ref{lem:vertices-to-cocycles}(i) shows that the image is in $\{[N'],~N'\in\llbracket N\rrbracket\}$. Theorem~\ref{thm:counting-cocycle-classes} is an immediate consequence of the following result.

\begin{thm} \label{thm:bijection-equiv-classes}
Let $N$ be a regular oriented matroid. The map $\ds \Psi:\Om_N^\geq\to \{[N'],~N'\in\llbracket N\rrbracket\}$ is a bijection. 
Moreover a map $\al\in \Om_N^\geq$ is in $\Om_N^>$ if and only if $\Psi(\al)$ is an acyclic equivalence class.
\end{thm}

\begin{proof}
The surjectivity and injectivity of $\Psi$ follows easily from Lemma~\ref{lem:vertices-to-cocycles}(i) and (ii) respectively. 
For the surjectivity, we consider a cocycle $S$ of $\un N$. By Lemma~\ref{lem:vertices-to-cocycles}(i), the point $x\in \{0,1/2\}^A$ such that $S(x)=S$ is a vertex in $V_N$. Hence $[{}_{-S}N]=[{}_{-S(x)}N]=\Psi(\al_x)$ is in the image of $\Psi$.
For the injectivity, consider $\al,\be\in\Om_N^\geq$ such that $\Psi(\al)=\Psi(\be)$. Let $x$ and $y$ be vertices of $\Pi_{N,\al}$ and $\Pi_{N,\be}$ respectively. By hypothesis, $[{}_{-S(x)}N]=\Psi(\al)=\Psi(\be)=[{}_{-S(y)}N]$, which by Lemma~\ref{lem:vertices-to-cocycles}(ii) implies $\al=\al_x=\al_y=\be$.



It remains to prove the statement about $\Om_N^>$. 
Let $\al\in\Om_N^>$. Let $x$ be a vertex of $\Pi_{N,\al}$ and let $S=S(x)$. By definition, $\Psi(\al)=[{}_{-S}N]$. We need to prove that ${}_{-S}N$ is acyclic. 
Since $\al$ is in $\Om_N^>$, there exists $y$ in $(0,1/2)^A\cap \Pi_{N,\al}$. 
By definition $\al_x=\al_y=\al$, hence the point $z:=y-x$ satisfies $\langle z,\ov C\rangle=\al_y(C)-\al_x(C)=0$ for all $C\in\fC$. 
Now consider the point $z'=(z_a')\in \RR^A$ where $z_a'=-z_a=1/2-y_a$ if $a\in S$ and $z_a'=z_a=y_a$ otherwise. 
By definition, the circuits of ${}_{-S}N$ are of the form ${{}_{-S}C}$ for $C\in \fC$,
and we compute $\ds \langle z',\overline{{}_{-S}C}\rangle =\langle z,\ov C\rangle =0.$ 
Thus, the point $z'$ is in $\Pi_{{}_{-S}N}^>$, and $\Om_{{}_{-S}N}^>$ contains the zero map. By Lemma~\ref{lem:AlphaZero}, this implies that ${}_{-S}N$ is indeed acyclic.

Lastly, consider $\al\in \Om_N^\geq$, and suppose that $x$ be a vertex of $\Pi_{N,\al}$ such that ${}_{-S(x)}N$ is acyclic. We want to prove that $\al\in \Om_N^>$. Let $S=S(x)$. Since ${}_{-S}N$ is acyclic, Lemma~\ref{lem:AlphaZero} ensures that the zero map belongs to $\Om_{{}_{-S}N}^>$. Hence there exists $z'\in (0,1/2)^A$ such that, $\ds \langle z',\overline{{}_{-S}C}\rangle =0$ for all 
$C\in \fC$. 
Let $z\in \RR^A$ be the point defined by $z_a'=-z_a$ if $a\in S$ and $z_a'=z_a$ otherwise. 
The point $z$ satisfies $\langle z,\ov C\rangle=\langle z',\overline{{}_{-S}C}\rangle=0$ for all $C\in\fC$, hence the point $y=x+z$ satisfies $\langle y,\ov C\rangle=\langle x,\ov C\rangle=\al(C)$. Thus, the point $y$ is in $\Pi_{N,\al}\cap (0,1/2)^A$, and the map $\al=\al_y$ is in $\Om_N^>$.
This completes the proof of Theorems~\ref{thm:bijection-equiv-classes} and~\ref{thm:counting-cocycle-classes}.
\end{proof}

\medskip

\section{Concluding Remarks}\label{Perspectives}

We conclude in this section with some directions for further research.

\subsection{Evaluations of $A_N$ and convolution formulas}
We have established some combinatorial interpretations for several evaluations of the $A$-polynomial and the related invariants $T_P^{(1)}$ and  $T_P^{(2)}$.
We wonder if other evaluations or specializations of these invariants have a meaningful interpretation. \ob{In particular, it would also be worth investigating whether the reciprocity results we have established (evaluations of $A_N(q,y,z)$ at negative integer values of $q$) can be interpreted in the Hopf-algebra framework of Aguiar and Ardila~\cite{Aguiar-Ardila:Hopf}, and whether additional reciprocity results can be harnessed through this framework.} We also wonder if the convolution formula for the Tutte polynomial~\cite{KRS:convolutionTutte} has a generalization for the invariants $T_P^{(1)}$ and  $T_P^{(2)}$. More generally, it would be interesting to see if other known properties of the Tutte polynomial have generalizations in terms of the $A$-polynomial. 

\subsection{Flow counting and another duality relation}\label{additional-duality}
\newcommand{\btau}{\bar{\tau}}
\newcommand{\wal}{\widehat{\al}}
\newcommand{\wbe}{\widehat{\be}}
\newcommand{\KK}{\mathbb{K}}
\ob{An anonymous referee pointed out to us another duality relation for the $A$-polynomial, related to the Fourier-transform duality between flows and coflows. 
Let us start by stating the result: for any regular oriented matroid $N=(A,\mfrk{C})$, 
\begin{equation}\label{eq:dualityq=3}
\Apoly_N(3,y,z)=\frac{(1+y+z)^{|A|}}{3^{|A|-\rk(N)}}\Apoly_{N^*}\bigg(3,\frac{1+\btau y+\tau z}{1+y+z},\frac{1+\tau y+\btau z}{1+y+z}\bigg),
\end{equation}
where $\tau=e^{2\pi i/3}$ and $\btau=e^{-2\pi i/3}$ are the primitive 3rd roots of unity.
}

\ob{In order to explain this result, we need to recall some background on the discrete Fourier transforms (equivalently, the representation of finite abelian groups). We refer the reader to~\cite{Biggs:InteractionModels,Terras:DFT} or~\cite[Section 2]{GarijoGoodall:flows} for more information.
Fix a positive integer $q$ and consider the additive group 
$$G=(\ZZ/q\ZZ)^A=\{u=(u_a)_{a\in A}\mid \forall a\in A,~ u_a\in \ZZ/q\ZZ\}.$$
We equip $G$ with the inner product defined by $\langle u, v\rangle:=\sum_{a\in A}u_av_a\in \ZZ/q\ZZ$ for all $u,v\in G$.
For a field $\KK$ containing $\CC$, and a function $\al:G\to \KK$, we define the \emph{discrete Fourier transform} $\wal:G \to \KK$ by 
$$\forall v\in G,~\wal(v):=\sum_{u\in G}\al(u)\tau^{-\langle u, v\rangle},$$
where $\tau=e^{2\pi i/q}$ is a $q$th root of unity. 
In the special case where there is a function $\be:\ZZ/q\ZZ\to \KK$ such that $\al(u)=\prod_{a\in A}\be(u_a)$ for all $u\in G$, the discrete Fourier transform is given by 
$$\wal(v)=\prod_{a\in A}\wbe(v_a),$$ 
where $\ds \wbe(k)=\sum_{j\in \ZZ/q\ZZ}\be(j)\tau^{-jk}.$
The discrete Fourier transform can be inverted, and enjoys a convolution formula from which one can derive the \emph{Poisson summation formula}.
Given a subgroup $H$ of $G$, one defines its \emph{orthogonal subgroup} as
$$H^\perp=\{u\in G\mid \forall v\in H,~\langle u, v\rangle=0 \},$$
and a special case of the Poisson summation formula gives
\begin{equation}\label{eq:poisson-sum}
\sum_{u\in H}\al(u)=\frac{1}{|H^\perp|}\sum_{v\in H^\perp}\wal(v).
\end{equation}
Note that the elements of $G$ can be identified with the set of functions $\{f:A\to \ZZ/q\ZZ\}$ (an element $u\in G$ is identified with the function $f:a\mapsto u_a$). Under this identification, the set of coflows of the oriented matroid $N$ corresponds to a subgroup $H_{N}$ of $G$ and the orthogonal subgroup $H_{N}^\perp=H_{N^*}$ corresponds to the coflows of the dual matroid $N^*$ (equivalently, the flows of $N$). We can therefore express the $A$-polynomial (at an odd positive integer $q$) in the form of the left-hand side of~\eqref{eq:poisson-sum}, by taking $H=H_{N}$ and $\al(u)=\prod_{a\in A}\be(u_a)$, where 
$$
\be(j):=\left|\begin{array}{ll} 1 &\textrm{ if }j=0, \\ y& \textrm{ if }j\in\{1,2,\ldots,(q-1)/2\},\\ z& \textrm{ if }j\in\{-1,-2,\ldots,-(q-1)/2\}.\end{array}\right.
$$
Equation \eqref{eq:poisson-sum} then gives
\begin{equation}\label{eq:poissonA}
\Apoly_N(q,y,z)=\sum_{f:A\to \ZZ/q\ZZ\textrm{ coflow of }N}~\prod_{a\in A}\be(f(a))=q^{\rk(N)-|A|}\sum_{g:A\to \ZZ/q\ZZ\textrm{ coflow of }N^*}~\prod_{a\in A} \wbe(g(a)),
\end{equation}
where $\ds \wbe(k):=\sum_{j\in \ZZ/q\ZZ}\be(j)\tau^{-j\,k}=1+\sum_{j=1}^{(q-1)/2}y \tau^{-jk}+z\tau^{jk}$.
Specializing~\eqref{eq:poissonA} at $q=3$ readily gives~\eqref{eq:dualityq=3}, showing that the $q=3$ evaluation of $\Apoly_{N}$ and $\Apoly_{N^*}$ are equivalent up to a change of variables. However for $q>3$, the function $\wbe$ is not constant over $\{1,2,\ldots,(q-1)/2\}$, hence the right-hand side of~\eqref{eq:poissonA} is not an evaluation of $\Apoly_{N^*}$. We wonder, along with the anonymous referee, whether there exists a natural invariant of regular oriented matroids which would generalize both $\Apoly_N$ and $\Apoly_{N^*}$.}

\subsection{Even section of the coflow polynomials}
\newcommand{\even}{\textrm{even}}
Another avenue for further investigation is the ``even part'' of the quasipolynomials studied in this article. Recall in particular that the weak and strict characteristic polynomials of a regular oriented matroid $N$ are defined in terms of some quasipolynomials $E^\geq_N(q)$  and $E^{>}_N(q)$ of period~2 defined via Ehrhart theory. Precisely, $\chi_N^{\geq}$ (resp. $\chi_N^{>}$) is the polynomial interpolating the evaluations of $E^\geq_N$ (resp. $E^{>}_N$) at odd integers. In Section~\ref{Q=0} we have shown that the evaluations of  $E^\geq_N(q)$  and $E^{>}_N(q)$ at $q=0$ count some equivalence classes of orientations: see Remark~\ref{rk:even-evaluation}. This raises the question of whether other evaluations of   $E^\geq_N(q)$  and $E^{>}_N(q)$ at even (non-positive) integers have meaningful interpretations. More generally, are there interesting properties of $N$ captured by the polynomials interpolating the evaluations of $E^\geq_N$  and $E^{>}_N$ at even integers.

The question can be lifted at the level of the $A$-polynomial. Indeed, by using a similar approach to the one developed in this paper, one can show that there exists a unique polynomial $\Apoly^\even_N(q,y,z,w)$ in the variables $q,y,z,w$ such that for all positive integer~$q$,
\begin{equation*}
\Apoly^\even_N(q,y,z,w)=\sum_{f\in F_N(q)}y^ {|f_A^ >|}\, z^{|f_A^ <|}\, w^{|f_A^{=\sfrac{q}{2}}|},
\end{equation*}
where $f_A^{=\sfrac{q}{2}}=\{a\in A\mid f(a)=q/2\}$.
One can also easily show that this invariant is related to the Tutte polynomial via the following relations. 
For any regular matroid $M=(E,\mC)$,
\begin{eqnarray*}
\Apoly^\even_{\orient{M}}(q,y^2,z^2,yz)&=&\Potts_M(q,y^2z^2),\\
\frac{1}{2^{|E|}}\sum_{\vec{M}\in \Orient(M)} \Apoly^\even_{\vec{M}}\left(q,y,z,\frac{y+z}{2}\right)&=&\Potts_M\!\left(q,\frac{y+z}{2}\right),
\end{eqnarray*}
and for any regular oriented matroid $N$,
\begin{eqnarray*}
\Apoly^\even_N(q,y,y,y) &=& \Potts_{\underline{N}}(q,y).
\end{eqnarray*}
It would be worth investigating which properties of $N$ are captured by the invariant~$\Apoly^\even_N$.




\subsection{A cyclically-symmetric refinement of the $A$-polynomial}\label{subsec:cyclically-sym}
Lastly, let us introduce a refinement of the $A$-polynomial of digraphs (in infinitely many variables). Let $D=(V,A)$ be a digraph.
Let $q$ be an odd positive integer, and let $\bu^q=\{u_k,~k\in \ZZ/q\ZZ\}$ be a set of variables indexed by the elements of $\ZZ/q\ZZ$. We define an invariant $\bbA_D^q(\bu^q,y,z)$ of $D$ as follows:
$$\bbA_D^q(\bu^q;y,z):=\sum_{f:V\to \ZZ/q\ZZ}\left(\prod_{v\in V}u_{f(v)}\right)y^{|f_A^>|}\,z^{|f_A^<|},$$
where $f_A^>=\{a=(u,v)\in A\mid f(v)-f(u)\in \{1,2,\ldots,(q-1)/2\}\}$ and $f_A^<=\{a=(u,v)\in A\mid f(u)-f(v)\in \{1,2,\ldots,(q-1)/2\}\}$. In other words,
$\bbA_D^q(\bu^q;y,z)$ counts the $q$-colorings $f$ of $D$ according the statistics $|f_A^>|$, $|f_A^<|$, and the number of times each color is used.
Comparing this with the expression~\eqref{eq:def-A-vertices} for $\Apoly_D(q,y,z)$, we see that the evaluation $\Apoly_D(q,y,z)$ of the $A$-polynomial at the integer $q$ is equal to $\frac{1}{q^{\comp(D)}}\bbA^q(\mathbf{1}^q;y,z)$, where $\mathbf{1}^q$ is obtained by setting $u_k=1$ for all $k\in  \ZZ/q\ZZ$. 

The polynomial $\bbA_D^q(\bu^q;y,z)$ is \emph{cyclically symmetric} in the variables $\bu^q$, in the following sense. Consider the set $R[\bu^q]$ of polynomials in the variables $\{u_k,~k\in \ZZ/q\ZZ\}$ with coefficients in a ring $R$. There is a natural action of the additive group $\ZZ/q\ZZ$ on $R[\bu^q]$ defined as follows: for $P\in R[\bu^q]$ and $\delta\in \ZZ/q\ZZ$, we define $\delta\cdot P$ as the polynomial in $R[\bu^q]$ obtained from $P$ by replacing the variable $u_k$ by $u_{k+\delta}$ for all $k\in \ZZ/q\ZZ$. We call a polynomial $P\in R[\bu^q]$ \emph{cyclically symmetric} if $\delta\cdot P=P$ for all $\delta \in \ZZ/q\ZZ$. It is clear from the definition that $\bbA_D^q(\bu^q;y,z)$ is cyclically symmetric since the statistics $|f_A^>|$ and $|f_A^<|$ are invariant when replacing the function $f$ by  $f+\delta$ (where $f+\delta$ is defined by $(f+\delta)(v)=f(v)+\delta$ for all $v\in V$).


We now define a polynomial $\bbA_D^\infty$ (in infinitely many variables) which contains all the invariants~$\bbA_D^q$. Let $\bt=\{t_r,~r\in\QQ/\ZZ\}$ be a set of variables indexed by the quotient group $\QQ/\ZZ$.  We define 
$$\bbA_D^\infty(\bt;y,z):=\sum_{f:V\to \QQ/\ZZ}\left(\prod_{v\in V}t_{f(v)}\right)y^{|f_A^>|}\,z^{|f_A^<|},$$
where $f_A^>=\{a=(u,v)\in A\mid f(v)-f(u)\in (0,1/2)+\ZZ\}$ and   $f_A^<=\{a=(u,v)\in A\mid f(u)-f(v)\in (0,1/2)+\ZZ\}$.
This invariant is \emph{cyclically-symmetric} in $\bt$ in the following sense: for all $\delta\in \QQ/\ZZ$ it is invariant by replacing each variable $t_r$ by $t_{r+\delta}$. We call $\bbA_D^\infty(\bt;y,z)$ the \emph{cyclically-symmetric $A$-polynomial of $D$}.

Note that for any odd positive integer $q$, one has
$$\bbA_D^\infty(\bt_q;y,z)=\bbA_D^q(\bu^q;y,z),$$
where $\bt_q$ is the specialization of $\bt$ obtained by setting $t_{r}=u_{k+q\ZZ}$ if $r=k/q+\ZZ$ for some $k\in \ZZ$, and $t_r=0$ otherwise. 
Hence the cyclically-symmetric $A$-polynomial $\bbA_D^\infty$ can be specialized to each of the invariant $\bbA_D^q$, and in particular to the evaluation $\Apoly_D(q,y,z)$ for any odd positive integer~$q$. Since these evaluations determine the $A$-polynomial $\Apoly_D$, it is clear that the cyclically-symmetric $A$-polynomial $\bbA_D^\infty$ determines the $A$-polynomial $\Apoly_D$.

Perhaps more surprisingly, the cyclically-symmetric $A$-polynomial also determines the $B$-polynomial of $D$, and its quasisymmetric refinement studied in~\cite{Awan-Bernardi:B-poly}. To see this, fix any infinite sequence $\ba=(a_n)_{n\in \ZZ^{>0}}$ of rational numbers in the interval $(0,1/2)$ such that $a_{n}<a_{n+1}$ for all $n$. Let $\xx=\{x_n,~n\in \ZZ^{>0}\}$ be a set of variables indexed by the positive integers and let $\bt_\ba$ be the specialization of $\bt$ obtained by setting $t_r=x_n$ if $r=a_n+\ZZ$ for some $n\in\ZZ^{>0}$, and $t_r=0$ otherwise. Then, it is easy to see that
$$\bbA_D^\infty(\bt_\ba;y,z)=\sum_{f:V\to \ZZ^{>0}}\left(\prod_{v\in V}x_{f(v)}\right)z^{|f_A^>|}z^{|f_A^<|},$$
where $f_A^>=\{a=(u,v)\in A\mid f(v)>f(u) \}$ and $f_A^<=\{a=(u,v)\in A\mid f(v)<f(u)\}$. The right-hand side is the \emph{quasisymmetric $B$-polynomial} $\bbB_D(\bx;y,z)$ defined in~\cite[Section 8]{Awan-Bernardi:B-poly}. It is a quasisymmetric function in the variables $\xx$ which is a refinement of the $B$-polynomial. Indeed, it is clear from the definition~\eqref{eq:defB} of the $B$-polynomial, that the evaluation of $\Bpoly_D(q,y,z)$ at a positive integer $q$ is obtained from $\bbB_D(\bx;y,z)$ by setting $x_n=1$ for $n\leq q$, and $x_n=0$ for $n>q$.

In~\cite{Awan-Bernardi:B-poly} we established various properties of the quasisymmetric $B$-polynomial. It would be interesting to study the cyclically-symmetric $A$-polynomial in the same manner.


\medskip

\noindent \textbf{Acknowledgments:} \ob{This main ideas for this article originated in 2016, when the first author was a graduate student at Brandeis university. We are grateful to Ira Gessel for several valuable discussions during this period. We are also grateful to the anonymous referees for their insightful comments, and for providing many additional references to the Tutte polynomial literature. We are particularly indebted to one of the referees for pointing out the duality relation presented in Section \ref{additional-duality}.}

\bibliographystyle{plain} 
\bibliography{biblio-Apoly}

\end{document}